\documentclass[a4paper]{article}

\usepackage[all]{xy}\usepackage[latin1]{inputenc}        %accents
\usepackage[dvips]{graphics,graphicx}
\usepackage{appendix}
\usepackage{amsfonts,amssymb,amsmath,xcolor,mathrsfs, amstext}
\usepackage{amsbsy, amsopn, amscd, amsxtra, amsthm,authblk, enumerate}
\usepackage{upref}
 \usepackage{geometry}
\usepackage[displaymath]{lineno}
\usepackage{float}

\usepackage[colorlinks,
            linkcolor=blue,
            anchorcolor=green,
            citecolor=blue
            ]{hyperref}

\numberwithin{equation}{section}

\def\e{\varepsilon}
\def\epsilon{\varepsilon}

\newcommand{\wt}{\widetilde}
\def\alb#1\ale{\begin{align*}#1\end{align*}}
\newcommand{\eqb}{\begin{equation}}
\newcommand{\eqe}{\end{equation}}

\DeclareMathOperator{\dist}{dist}

\newcommand{\bbD}{\mathbb{D}}

\newcommand{\bbH}{\mathbb{H}}

\newcommand{\bbR}{\mathbb{R}}
\newcommand{\bbP}{\mathbb{P}}

\newcommand{\SLE}{\mathrm{SLE}}

\newtheorem{theorem}{Theorem}[section]
\newtheorem*{thma}{Theorem A}
\newtheorem*{thmb}{Theorem B}
\newtheorem{lemma}[theorem]{Lemma}

\newtheorem{proposition}[theorem]{Proposition}
\newtheorem*{proposition*}{Proposition}
\newtheorem{corollary}{Corollary}[section]
\newtheorem*{corollary*}{Corollary}
\newtheorem{definition}[theorem]{Definition}
\newtheorem*{definitions*}{Definitions}

\newtheorem*{example*}{\bf Example}
\theoremstyle{remark}

\numberwithin{equation}{section}
\title{Time-reversal of multiple-force-point chordal $\SLE_\kappa(\underline{\rho})$}

\author{Pu Yu}
\date{\today}

\begin{document}

\maketitle

\begin{abstract}
Chordal SLE$_\kappa(\underline{\rho})$ is a natural variant of {the} chordal SLE curve. It is a family of random non-crossing curves on the upper half plane   from 0 to $\infty$, whose law is influenced by additional force points on $\mathbb R$. When there are force points away from the origin, the law of SLE$_\kappa(\underline{\rho})$ is not reversible,  unlike the ordinary chordal SLE$_\kappa$.   Zhan (2019) gives an explicit description of the law of the time reversal of  
SLE$_\kappa(\underline{\rho})$ when all force points lie on the same sides of the origin, and conjectured that a similar result holds in general. 
We prove his conjecture. {Specifically}, based on Zhan's result, using the techniques from the Imaginary Geometry developed by Miller and Sheffield (2013), we show that when $\kappa\in(0,8)$, the law of the time reversal of non-boundary filling $\SLE_\kappa(\underline{\rho})$ process is absolutely continuous with respect to $\SLE_\kappa(\underline{\hat{\rho}})$ for some $\underline{\hat{\rho}}$ determined by $\underline{\rho}$, with the Radon-Nikodym derivative being a product of conformal derivatives.
\end{abstract}
	\section{Introduction}

The Schramm-Loewner Evolution (SLE$_\kappa$) with $\kappa>0$  is an important 
family of random non-self-crossing curves introduced by Schramm~\cite{Sch00}.
They have been proved or conjectured to described a large class of two-dimensional lattice models at criticality.  
We refer the reader to~\cite{Law08,Sch06ICM,SmirnovICM} for basic properties of SLE and their relation to 2D lattice models.

The most basic version of SLE is the chordal SLE$_\kappa$ curve, which is a random curve between two boundary points of a {simply} connected domain characterized by conformal invariance and the domain Markov property.
It was conjectured by Rohde and Schramm~\cite{RS05} that chordal SLE$_\kappa$ with $\kappa\in (0,8]$
satisfies reversibility. Namely, modulo a time reparametrization the time reversal of a chordal SLE$_\kappa$ curve  is also a chordal SLE$_\kappa$.
The conjecture was first proved for $\kappa\in (0,4]$ by Zhan~\cite{Zhan08reverse} using the so-called  commutation coupling. 
The $\kappa\in (4,8)$ case was proved by Miller and Sheffield~\cite{IGIII}   using the imaginary geometry theory. {The chordal $\SLE_8$ is
	the scaling limit of UST Peano curve with half free and half wired boundary
	conditions~\cite{lawler2011conformal} and therefore is also reversible.}  

Chordal SLE$_\kappa(\underline{\rho} )$ curves are important variants of chordal SLE. They are  still curves between two boundary points of a {simply} connected domain, but their laws depend on  some additional marked points called force points. 
They were introduced by Lawler, Schramm and Werner~\cite{LSW03conformalrestriction} in the theory of conformal restriction, and play a fundamental role in imaginary geometry as flow lines emanating from a boundary point~\cite{MS16a}. In~\cite{MS16b,IGIII}, it was proved that chordal SLE$_\kappa(\underline{\rho} )$ for $\kappa\in(0,8)$ with at most two force points lying infinitesimally close to the starting point satisfy the reversibility.   When there are force points away from the origin, the law of chordal SLE$_\kappa(\underline{\rho})$ is not reversible anymore. Recently, Zhan~\cite{Zhan19} gave an explicit description of the law of  the time reversal of  
SLE$_\kappa(\underline{\rho})$ when $\kappa\in (0,4]$ and all force points lie on the same side of the origin, and when $\kappa\in (4,8)$, all force points lie on the same side, and the curve is not boundary touching on this side. In the same paper, he conjectured that a similar result holds for general 
chordal SLE$_\kappa(\underline{\rho})$ with $\kappa\in (0,8)$ as long as the curve is non-boundary filling; see~\cite[Conjecture 1.3]{Zhan19}. In this paper we prove his conjecture. 

To state our main result, we introduce the necessary notations to describe chordal SLE$_\kappa(\underline{\rho} )$ curves with their precise definition postponed to Section~\ref{subsec:pre-ig}.
Let {$\kappa\in (0,8]$}. Fix the force points $x^{k,L}<...<x^{1,L}<x^{0,L}=0^-<x^{0,R}=0^+<x^{1,R}<...<x^{\ell,R}$ and for each force point $x^{i,q}$, $q\in\{L,R\}$, we assign a weight $\rho^{i,q}\in\bbR$, such that 
\begin{equation}\label{eq:continuation-threshold}
	\sum_{i=0}^j\rho^{i,L}>(-2)\vee(\frac{\kappa}{2}-4) \ \text{for all}\ 0\le j\le k\ \ \text{and} \ \ \sum_{i=0}^j\rho^{i,R}>(-2)\vee(\frac{\kappa}{2}-4)  \ \text{for all}\ 0\le j\le \ell.  
\end{equation}
We refer to the {vectors of force points and weights}  as $\underline{x} = (\underline{x}^L;\underline{x}^R)$ and $\underline{\rho} = (\underline{\rho}^L;\underline{\rho}^R)$. Given an $\SLE_\kappa(\underline{\rho})$ process $\eta$ from 0 to $\infty$ in the upper half plane $\bbH$ with force points $\underline{x}$, for each $i\ge 1$ and $q\in\{L,R\}$, let $D_\eta^{i,q}$ be the connected component of $\bbH\backslash\eta$ containing $x^{i,q}$, and $\sigma_\eta^{i,q},\xi_\eta^{i,q}$ be the first and the last point on $\partial D_\eta^{i,q}$ traced by $\eta$.  Consider the conformal map $\psi_\eta^{i,q}:D_\eta^{i,q}\to\bbH$ sending {$(\sigma_\eta^{i,q},x_\eta^{i,q},\xi_\eta^{i,q})$} to $(0,\pm1,\infty)$ where we take the $+$ sign  when $q=R$ and take the $-$ sign  when $q=L$. 

We now introduce a family of  measures on curves describing the time reversal of chordal $\SLE_\kappa(\underline{\rho})$.
\begin{definition}\label{def:sle-weight-confradius}
	{Suppose $\underline{x}$ and $\underline{\rho}$  satisfy~\eqref{eq:continuation-threshold}}. We associate a power parameter $\alpha^{i,q}\in\bbR$ for each $x^{i,q}$ with $\alpha^{0,L}=\alpha^{0,R}=0$.
	Define  $\wt{\SLE}_\kappa(\underline{\rho};\underline{\alpha})$ {with force points $\underline{x}$} to be the measure on {continuous curves in $\overline{\bbH}$ from 0 to $\infty$} which is absolutely continuous with respect to $\SLE_\kappa(\underline{\rho})$ with Radon-Nikodym derivative
	\begin{equation}\label{eq:def-sle-weight-confradius}
		\frac{d\wt{\SLE}_\kappa(\underline{\rho};\underline{\alpha})}{d\SLE_\kappa(\underline{\rho})}(\eta) = \prod_{q\in\{L,R\}}\prod_{i\ge 1} |x^{i,q}\cdot(\psi_\eta^{i,q})'(x^{i,q})|^{\alpha^{i,q}}.  
	\end{equation}
\end{definition}
\begin{figure}
	\centering
	\includegraphics[scale=0.78]{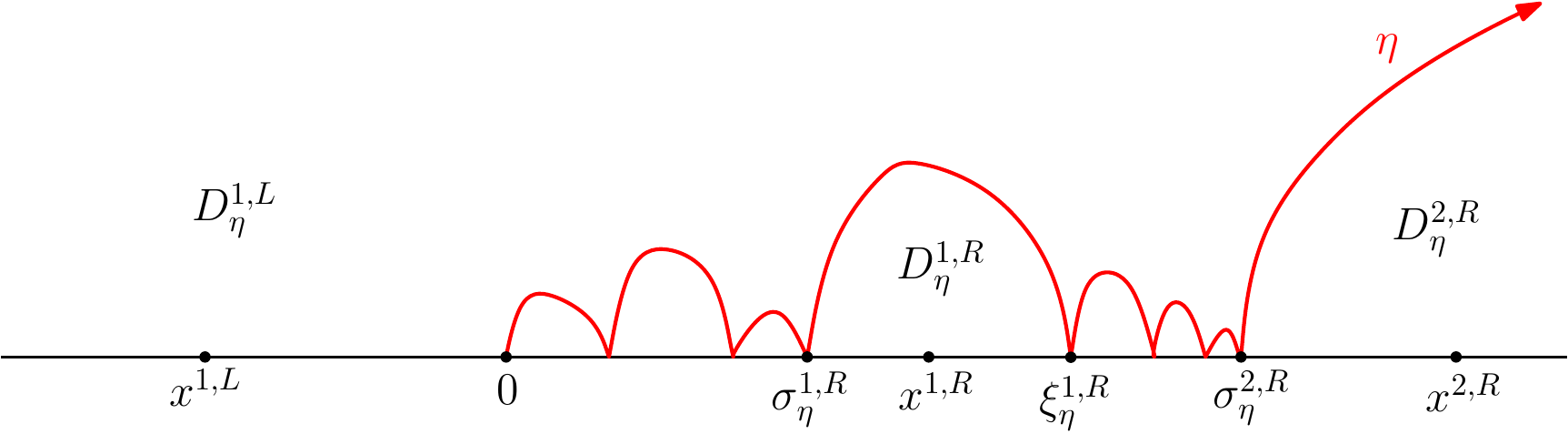}
	\caption{An   $\wt{\SLE}_\kappa(\underline{\rho};\underline{\alpha})$ processes with force points $x^{1,L},0^-;0^+,x^{1,R},x^{2,R}$. By definition $\sigma_\eta^{1,L}=0^-$, $\xi_\eta^{1,L}=\xi_\eta^{2,R}=\infty$.  For $i=1,2$, the conformal map $\psi_\eta^{i,R}:D_\eta^{i,R}\to\bbH$ sends $(\sigma_\eta^{i,R},x^{i,R},\xi^{i,R})$ to $(0,1,\infty)$, and the conformal map $\psi_\eta^{1,L}:D_\eta^{1,L}$ sends $(\sigma_\eta^{1,L},x^{1,L},\xi^{1,L})$ to $(0,-1,\infty)$. Then the measure $\wt{\SLE}_\kappa(\rho^{0,L},\rho^{1,L};\\ \rho^{0,R},\rho^{1,R},\rho^{2,R};\alpha^{1,L};\alpha^{1,R},\alpha^{2,R})$ is absolutely continuous w.r.t.\  ${\SLE}_\kappa(\rho^{0,L},\rho^{1,L};\rho^{0,R},\rho^{1,R},\rho^{2,R})$ with Radon-Nikodym derivative $|x^{1,L}(\psi_\eta^{1,L})'(x^{1,L})|^{\alpha^{1,L}}|x^{1,R}(\psi_\eta^{1,R})'(x^{1,R})|^{\alpha^{1,R}}|x^{2,R}(\psi_\eta^{2,R})'(x^{2,R})|^{\alpha^{2,R}}$. 
}
	\label{fig:reverse-def}
\end{figure}

Let us recall some statements on the time reversal of chordal $\SLE_\kappa(\underline{\rho})$ processes from existing literature. The first one is about the time reversal of $\SLE_\kappa(\rho^L;\rho^R)$ processes, which is shown in~\cite[Theorem 1.1]{MS16b} and~\cite[Theorem 1.2]{IGIII}. Let $J:\bbH\to\bbH$ be the map $J(z)=-1/z$. {For a curve $\eta$, we write $\mathcal{R}(\eta)$ for its time reversal.} 
\begin{thma}\label{thm:igII}
	Let $\kappa\in(0,8]$ and $\rho^L,\rho^R>-2$ such that  $\rho^L,\rho^R\ge\frac{\kappa}{2}-4$ if $\kappa\in (4,8]$. Let $\eta$ be an $\SLE_\kappa(\rho^L;\rho^R)$ process in $\bbH$ from 0 to $\infty$ with force points $0^-;0^+$. Then {modulo time parametrization,} $\mathcal{R}(J\circ\eta)$ is the $\SLE_\kappa(\rho^R;\rho^L)$ process in $\bbH$ from 0 to $\infty$ with force points $0^-;0^+$.
\end{thma}
We comment that the $\kappa=8$ case above is not stated in~\cite[Theorem 1.2]{IGIII}, yet it readily follows from the reversibility of chordal $\SLE_8$ and~\cite[Theorem 1.1]{MS16b} along with SLE duality~\cite{zhan2008duality,MS16a} (see Proposition~\ref{prop:duality}). 

When all the force points lie on the same side of 0, the following theorem is {shown in Theorem 1.2 and Section 3.2 of~\cite{Zhan19} via the construction of the reversed curve}. 
\begin{thmb}\label{thm:dapeng}
	Let $\ell\ge0$. Fix $\rho^{0,R},...,\rho^{\ell,R}\in\bbR$, such that $\kappa\in(0,4]$, $\min_{0\le j\le\ell}\sum_{i=0}^j\rho^{i,R}>-2$ if $\kappa\in(0,4]$, and  $\min_{0\le j\le\ell}\sum_{i=0}^j\rho^{i,R}\ge\frac{\kappa}{2}-2$ if $\kappa\in(4,8)$. Let $\underline{\rho}^R$ be the vector of $\rho^{i,R}$ and $\eta$ be a chordal $\SLE_\kappa(\underline{\rho}^R)$ curve in $\bbH$ from $0$ to $\infty$ with force points $0^+=x^{0,R}<x^{1,R}<...<x^{\ell,R}$.   Let $x^{\ell+1,R}=+\infty$ and $\rho^{\ell+1,R} = -\sum_{i=0}^\ell \rho^{i,R}$.   For $0\le i\le \ell$, let $\hat{x}^{i,L} = J(x^{\ell+1-i,R})$, $\hat{\rho}^{i,L}=-\rho^{\ell+1-i,R}$.
	Here we use the convention $J(\pm\infty)=0^{\mp}$. For $1\le i\le \ell$, let $\hat{\alpha}^{i,L} = \frac{\hat{\rho}^{i,L}(\kappa-4)}{2\kappa}$. Let $\underline{\hat{x}}^L$, $\underline{\hat{\rho}}^L$, $\underline{\hat{\alpha}}^L$ be the vector of $\hat{x}^{i,q}$, $\hat{\rho}^{i,q}$ and $\hat{\alpha}^{i,q}$. Then up to reparametrization, the law of {$\mathcal{R}(J\circ\eta)$} is equal to $\frac{1}{Z}\wt{\SLE}_\kappa(\underline{\hat{\rho}}^L;\underline{\hat{\alpha}}^L)$ with force points $\underline{\hat{x}}^L$ for some normalizing constant $Z\in(0,\infty)$.
\end{thmb}

In~\cite{Zhan19}, {the time reversal of $\SLE_\kappa(\underline{\rho}^R)$ is described in terms of reversed intermediate $\SLE_\kappa(\underline{\rho})$ ($\mathrm{i}\SLE_\kappa^r(\underline{\rho})$) process, which agrees with $\wt{\SLE}_\kappa(\underline{\hat{\rho}}^L;\underline{\hat{\alpha}}^L)$ when normalized to be a probability measure. The $\mathrm{i}\SLE_\kappa^r(\underline{\rho})$} 
process is described explicitly using a Loewner evolution based on  Appell-Lauricella multiple hypergeometric function.    The constant $Z$ can be traced via~\cite[(3.16),(3.19)]{Zhan19} and~\cite[Remark 3.6]{Zhan19}, and can be expressed by a hypergeometric function (in fact a product of the gamma functions) depending only on $\kappa,\underline{\rho}^R$ but not on the location of the force points $\underline{x}^R$. % We will show that this independence of $\underline{x}^R$ also holds in some way when we are adding a single force point at $0^-$.

Our main result is the following.
\begin{theorem}\label{thm:main}
	{Let $\kappa\in(0,8]$}. Fix $\underline{x}$, $\underline{\rho}$ with~\eqref{eq:continuation-threshold}, and let $\eta$ be a chordal $\SLE_\kappa(\underline{\rho})$ curve in $\bbH$ from $0$ to $\infty$ with force points $\underline{x}$. Let $x^{k+1,L}=-\infty$, $x^{\ell+1,R}=+\infty$, $\rho^{k+1,L} = -\sum_{i=0}^k\rho^{i,L}$ and $\rho^{\ell+1,R} = -\sum_{i=0}^\ell \rho^{i,R}$.    For $0\le i\le \ell$, let $\hat{x}^{i,L} = J(x^{\ell+1-i,R})$, $\hat{\rho}^{i,L}=-\rho^{\ell+1-i,R}$.
	For $0\le i\le k$, let $\hat{x}^{i,R} = J(x^{k+1-i,L})$, $\hat{\rho}^{i,R}=-\rho^{k+1-i,R}$. For $i\ge 1$ and $q\in\{L,R\}$, let $\hat{\alpha}^{i,q} = \frac{\hat{\rho}^{i,q}(\kappa-4)}{2\kappa}$. Let $\underline{\hat{x}}$, $\underline{\hat{\rho}}$, $\underline{\hat{\alpha}}$ be the vector of $\hat{x}^{i,q}$, $\hat{\rho}^{i,q}$ and $\hat{\alpha}^{i,q}$. Then up to reparametrization, the law of {$\mathcal{R}(J\circ\eta)$} is {equal to $\frac{1}{Z}\wt{\SLE}_\kappa(\underline{\hat{\rho}};\underline{\hat{\alpha}})$ with force points $\underline{\hat{x}}$ for some normalizing constant $Z:=Z(\underline{\rho})\in(0,\infty)$}.
\end{theorem}

\begin{figure}
	\centering
	\includegraphics[scale=0.78]{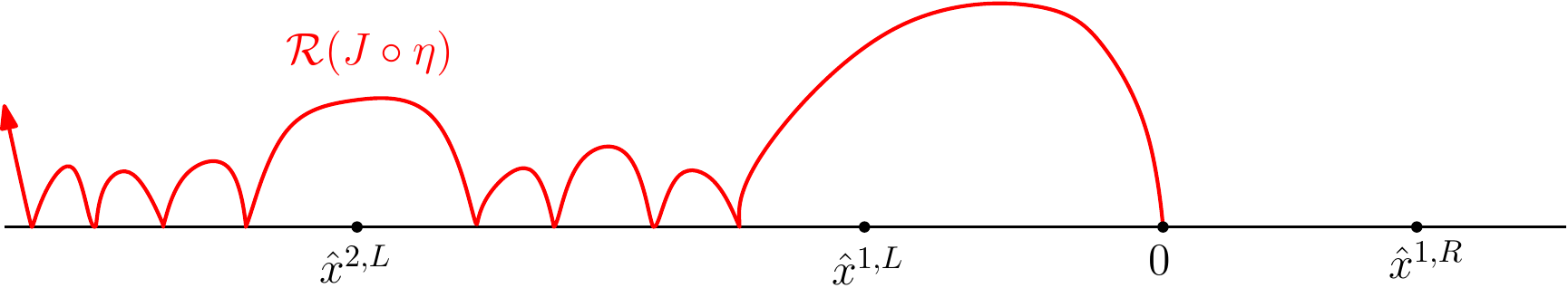}
	\caption{An example of Theorem~\ref{thm:main}. Let $\alpha^{i,q}=0$ in Figure~\ref{fig:reverse-def} so that $\eta$ is an ${\SLE}_\kappa(\rho^{0,L},\rho^{1,L}; \rho^{0,R},\rho^{1,R},\\\rho^{2,R})$ process with force points $0^-,x^{1,L},0^+,x^{1,R},x^{2,R}$. Let $J(z)=-1/z$, $\hat{x}^{1,L}=J(x^{2,R})$, $\hat{x}^{2,L}=J(x^{1,R})$, $\hat{x}^{1,R}=J(x^{1,L})$. Then the time reversal of $J\circ\eta$ is an  $\wt{\SLE}_\kappa(\rho^{0,R}+\rho^{1,R}+\rho^{2,R},-\rho^{2,R},-\rho^{1,R};\rho^{0,L}+\rho^{1,L},-\rho^{1,L};\frac{\rho^{2,R}(4-\kappa)}{2\kappa},\frac{\rho^{1,R}(4-\kappa)}{2\kappa};\frac{\rho^{1,L}(4-\kappa)}{2\kappa})$ process with force points $0^-,\hat{x}^{1,L},\hat{x}^{2,L},0^+,\hat{x}^{1,R}$. When $\rho^{1,L}=\rho^{1,R}=\rho^{2,R}=0$, the law of $J\circ\eta$ is ${\SLE}_\kappa(\rho^{0,R};\rho^{0,L}  )$ with force points at $0^-$ and $0^+$, as shown in~\cite{MS16b}.}
	\label{fig:reverse-thm}
\end{figure}

The $\kappa=4$ case of Theorem~\ref{thm:main} is covered by~\cite[Theorem 1.1.6]{wang2017level} by realizing $\SLE_4(\underline{\rho})$ curves as level lines of Gaussian free field with appropriate boundary conditions. In this case, the reversed curve is {just ${\SLE}_\kappa(\underline{\hat{\rho}})$ with no weighting, i.e.,  $\wt{\SLE}_\kappa(\underline{\hat{\rho}};\underline{\hat{\alpha}})$ with $\hat{\alpha}^{i,q}=0$.} 

Based on Theorem~\hyperref[thm:igII]{A} and Theorem~\hyperref[thm:dapeng]{B}, our proof is mainly relying on the techniques from the Imaginary Geometry~\cite{MS16a,MS16b,IGIII,MS17}. For $\kappa\in(0,4)$, we first extend a commutation relation between two $\SLE_\kappa(\underline{\rho})$-type  processes {(i.e., two $\SLE_\kappa(\underline{\rho})$ with possibly different $\underline{\rho}$ values)} from the theory of GFF flow lines to the setting of two $\wt{\SLE}_\kappa(\underline{\rho};\underline{\alpha})$-type processes (Proposition~\ref{prop:commutation}), from which we are able to add a force point located at $0^-$ in Theorem~\hyperref[thm:dapeng]{B} (Lemma~\ref{lem:onesided+}). Using this extended result with the commutation relation,  we can construct a pair of $\wt{\SLE}_\kappa(\underline{\rho};\underline{\alpha})$-type processes $(\eta_1,\eta_2)$, such that conditioned on one curve, the time reversal of the other curve is the ordinary $\SLE_\kappa(\underline{\rho})$ process with only one degenerate force point (i.e.\ $0^{\pm}$) on the left or right side. Then from the SLE resampling property~\cite[Theorem 4.1]{MS16b}, the two conditional laws uniquely characterize the joint law of the reversal of $(\eta_1,\eta_2)$, which finishes the proof for $\kappa\in(0,4)$. {For $\kappa\in(4,8]$}, we apply the $\kappa\in (0,4)$ result along with the SLE duality~\cite{zhan2008duality,Dub09,MS16a}, which states that for $\kappa>4$, the boundaries of $\SLE_\kappa$-type processes are $\SLE_{\frac{16}{\kappa}}$-type processes (see Proposition~\ref{prop:duality}). %In fact, this also shows that Theorem~\ref{thm:main} for $\kappa=8$ although we assumed $\kappa\in (0,8)$ to align with~\cite[Conjecture 1.3]{Zhan19}.
%The bulk of our proof of Theorem~\ref{thm:main} is devoted to the $\kappa\in (0,4)$ case. The $\kappa\in (4,8)$  case is an immediate  corollary of the $\kappa\in (0,4)$ result along with the SLE duality~\cite{zhan2008duality,Dub09,MS16a} (see Proposition~\ref{prop:duality}). In fact, this also shows that Theorem~\ref{thm:main} for $\kappa=8$ although we assumed $\kappa\in (0,8)$ to align with~\cite[Conjecture 1.3]{Zhan19}.

{For the $\kappa>8$ regime,  it has been shown in~\cite[Theorem 1.19]{MS17} that the time reversal of chordal $\SLE_\kappa(\rho^L;\rho^R)$ with force points at $0^-;0^+$ is $\SLE_\kappa(\tilde\rho^R;\tilde\rho^L)$ where $\rho^L,\rho^R\in (-2,\frac{\kappa}{2}-2)$ and $\tilde\rho^q = \frac{\kappa}{4}-2-\rho^q$ for $q\in\{L,R\}$. The time reversal of $\SLE_\kappa(\underline{\rho})$ is not known when $\kappa>8$ and there are force points located at $\bbR\backslash \{0\}$.}

We comment that the reversibility of $\SLE$ processes can also be {inferred from} the conformal welding of Liouville quantum gravity surfaces (see e.g.~\cite{DMS14, AHS20, ASY22}). For instance, by viewing the welding interface from the opposite direction, Theorem~\hyperref[thm:igII]{A} is a direct consequence of~\cite[Theorem 2.2]{AHS20}. The time reversal of $\SLE_\kappa(\rho^-;\rho^+,\rho_1)$ with force points $0^-;0^+,1$ has also been discussed in~\cite[Section 7.1]{ASY22} via the conformal welding of quantum triangles. We expect that this method can also be used to describe the time reversal of other types of SLE curves, such as radial SLE with force points and SLE on the annulus.

In Section~\ref{subsec:pre-ig}, we recap the $\SLE_\kappa(\underline{\rho})$ processes along with its coupling with the GFF as imaginary geometry flow lines in~\cite{MS16a}. In Section~\ref{subsec:pre-resample}, we establish  a commutation relation for $\wt{\SLE}_\kappa(\underline{\rho};\underline{\alpha})$ processes and recap the \emph{$\SLE$ resampling properties}. Finally in Section~\ref{sec:proof}, we prove Theorem~\ref{thm:main}.

\medskip

%%%%%%%%%%%%%%%%%%%%%%%%%%%%%%%%%%%%%%%%%%%%%%%%%%%%%%%%%%%%%%%%%%%
%%                                                               %%
%% No need for \maketitle.                                       %%
%%                                                               %%
%%%%%%%%%%%%%%%%%%%%%%%%%%%%%%%%%%%%%%%%%%%%%%%%%%%%%%%%%%%%%%%%%%%

%%%%%%%%%%%%%%%%%%%%%%%%%%%%%%%%%%%%%%%%%%%%%%%%%%%%%%%%%%%%%%%%%%%
%%                                                               %%
%% Please replace what follows by the body of your article       %%
%% (up to the bibliography):                                     %%
%%                                                               %%
%%%%%%%%%%%%%%%%%%%%%%%%%%%%%%%%%%%%%%%%%%%%%%%%%%%%%%%%%%%%%%%%%%%
\section{Preliminaries}\label{sec:pre}

In this paper we work with non-probability measures and extend the terminology of ordinary probability to this setting. For a finite or $\sigma$-finite  measure space $(\Omega, \mathcal{F}, M)$, we say $X$ is a random variable if $X$ is an $\mathcal{F}$-measurable function with its \textit{law} defined via the push-forward measure $M_X=X_*M$. In this case, we say $X$ is \textit{sampled} from $M_X$ and write $M_X[f]$ for $\int f(x)M_X(dx)$. \textit{Weighting} the law of $X$ by $f(X)$ corresponds to working with the measure $d\tilde{M}_X$ with Radon-Nikodym derivative $\frac{d\tilde{M}_X}{dM_X} = f$, and \textit{conditioning} on some event $E\in\mathcal{F}$ (with $0<M[E]<\infty$) refers to the probability measure $\frac{M[E\cap \cdot]}{M[E]} $  over the space $(E, \mathcal{F}_E)$ with $\mathcal{F}_E = \{A\cap E: A\in\mathcal{F}\}$. 

Throughout this paper, for a continuous simple curve $\eta$ from 0 to $\infty$ in ${\bbH}\cup\bbR$, we shall refer to the {subset} of $\bbH\backslash\eta$ consisted of connected components whose boundaries contain a subinterval of $(-\infty,0)$ (resp.\ $(0,\infty)$) as the left (resp.\ right) {part} of $\bbH\backslash\eta$. For $n\ge0$, $\underline{x} = (x_0,...,x_n)\in\bbR^{n+1}$ and $a\in\bbR$, {we write $a+\underline{x}$ for $(a+x_0,x_1,...,x_n)$ and $a\underline{x}$ for $(ax_0,...,ax_n)$.}  {The formal notation is used for weights and latter is for the locations of force points under dilation and $\SLE$ duality purposes (see Proposition~\ref{prop:duality}).}

\subsection{$\SLE_\kappa(\underline{\rho})$ process and the imaginary geometry}\label{subsec:pre-ig}
Fix $\kappa>0$. We start with the $\SLE_\kappa$ process on the upper half plane $\bbH$. Let $(B_t)_{t\ge0}$ be the standard Brownian motion. The $\SLE_\kappa$ is the probability measure on continuously growing {curves $(K_t)_{t\ge0}$ in $\overline{\bbH}$}, {whose mapping out function $(g_t)_{t\ge0}$ (i.e., the unique conformal transformation from $\mathbb{H}\backslash K_t$ to $\mathbb{H}$ such that $\lim_{|z|\to\infty}|g_t(z)-z|=0$) can be described by}
\begin{equation}\label{eq:def-sle}
	g_t(z) = z+\int_0^t \frac{2}{g_s(z)-W_s}ds, \ z\in\mathbb{H},
\end{equation} 
where $W_t=\sqrt{\kappa}B_t$ is the Loewner driving function. For the force points $x^{k,L}<...<x^{1,L}<x^{0,L}= 0^-<x^{0,R}= 0^+< x^{1,R}<...<x^{\ell, R}$ and the weights $\rho^{i,q}\in\bbR$, the $\SLE_\kappa(\underline{\rho})$ process is the probability measure on {curves  $(K_t)_{t\ge 0}$  in $\overline{\bbH}$} growing the same as ordinary SLE$_\kappa$ (i.e, satisfies \eqref{eq:def-sle}) except that the Loewner driving function $(W_t)_{t\ge 0}$ are now characterized by 
\begin{equation}\label{eq:def-sle-rho}
	\begin{split}
		&W_t = \sqrt{\kappa}B_t+\sum_{q\in\{L,R\}}\sum_i \int_0^t \frac{\rho^{i,q}}{W_s-V_s^{i,q}}ds; \\
		& V_t^{i,q} = x^{i,q}+\int_0^t \frac{2}{V_s^{i,q}-W_s}ds, \ q\in\{L,R\}.
	\end{split}
\end{equation}
It has been proved in \cite{MS16a} that the SLE$_\kappa(\underline{\rho})$ process a.s. exists, is unique and generates a continuous curve until the \textit{continuation threshold}, the first time $t$ such that $W_t = V_t^{j,q}$ with $\sum_{i=0}^j\rho^{i,q}\le -2$ for some $j$ and $q\in\{L,R\}$. % Let $f_t := g_t-W_t$ be the \textit{centered Loewner flow.}

Now we recap the definition of the Gaussian Free Field. Let $D\subsetneq \mathbb{C}$ be a domain. We construct the GFF on $D$ with \textit{Dirichlet} \textit{boundary conditions}  as follows. Consider the space of smooth functions on $D$ with finite Dirichlet energy and zero value near $\partial D$, and let $H(D)$ be its closure with respect to the inner product $(f,g)_\nabla=\int_D(\nabla f\cdot\nabla g)\ dx\ dy$. Then the (zero boundary) GFF on $D$ is defined by 
\begin{equation}\label{eqn-def-gff}
	h = \sum_{n=1}^\infty \xi_nf_n
\end{equation}
where $(\xi_n)_{n\ge 1}$ is a collection of i.i.d. standard Gaussians and $(f_n)_{n\ge 1}$ is an orthonormal basis of $H(D)$. The sum \eqref{eqn-def-gff} a.s.\ converges to a random distribution independent of the choice of the basis $(f_n)_{n\ge 1}$. For a function $g$ defined on $\partial D$ with harmonic extension $f$ in $D$ and a zero boundary GFF $h$, we say that $h+f$ is a GFF on $D$ with boundary condition specified by $g$. 
See \cite[Section 4.1.4]{DMS14} for more details. 

Next we introduce the notion of \emph{GFF flow lines}. We restrict ourselves to the range $\kappa\in(0,4)$.  Heuristically, given a GFF $h$, $\eta(t)$ is a flow line of angle $\theta$ if
\begin{equation}
	\eta'(t) = e^{i(\frac{h(\eta(t))}{\chi}+\theta)}\ \text{for}\ t>0, \ \text{where}\ \chi = \frac{2}{\sqrt{\kappa}}-\frac{\sqrt{\kappa}}{2}.
\end{equation} 
To be more precise, let $(K_t)_{t\ge 0}$ be the hull at time $t$ of the SLE$_\kappa(\underline{\rho})$ process described by the Loewner flow \eqref{eq:def-sle} with $(W_t, V_t^{i,q})$ solving \eqref{eq:def-sle-rho}, and let $\mathcal{F}_t$ be the filtration generated by  $(W_t, V_t^{i,q})$. %   introduces an exact coupling of a Dirichlet GFF with an $\SLE_\kappa(\underline{\rho})$, which we briefly recap as follows. with filtration $\mathcal{F}_t$. 
Let $\mathfrak{h}_t^0$ be the bounded harmonic function on $\mathbb{H}$ with boundary values 
\begin{equation}\label{eq:flowlinecouple}
	-\lambda(1+\sum_{i=0}^j \rho^{i,L})\  \text{on} \ (V_t^{j+1, L}, V_t^{j,L}),\ \ \ \text{and}\ \ \lambda(1+\sum_{i=0}^j \rho^{i,R})\ \text{on}\ \ (V_t^{j, R}, V_t^{j+1,R})
\end{equation}
and {$-\lambda$ on $(V_t^{0,L},W_t)$,  $\lambda$ on $(W_t,V_t^{0,R})$} where $\lambda = \frac{\pi}{\sqrt{\kappa}}$, $x^{k+1, L} = -\infty, x^{\ell+1, R} = +\infty$. Set {$\mathfrak{h}_t(z) = \mathfrak{h}_t^0(g_t(z))-\chi\arg g_t'(z)$}. Let  $\tilde{h}$ be a zero boundary GFF on $\mathbb{H}$ and 
\begin{equation}\label{eq:ig-gff}
	h = \tilde{h}+ \mathfrak{h}_0.  
\end{equation}
Then as proved in \cite[Theorem 1.1]{MS16a}, there exists a coupling between $h$ and the  SLE$_\kappa(\underline{\rho})$ process $(K_t)$, such that for any $\mathcal{F}_t$-stopping time $\tau$ before the continuation threshold, $K_\tau$ is a local set for $h$ and the conditional law of $h|_{\mathbb{H}\backslash K_\tau}$ given $\mathcal{F}_\tau$ is the same as the law of {$\mathfrak{h}_\tau+\tilde{h}\circ g_\tau$}.

For $\kappa<4$, the SLE$_\kappa(\underline{\rho})$ coupled with the GFF $h$ as above is referred as a \emph{flow line} of $h$ from 0 to $\infty$, and we say an SLE$_\kappa(\underline{\rho})$ curve is a flow line of angle $\theta$ if it can be coupled with $h+\theta\chi$ in the above sense. For $\kappa'>4$, the SLE$_{\kappa'}(\underline{\rho})$ curve coupled with a GFF $-h$ as above is referred as a \text{counterflow lines} of $h$. 

\begin{figure}
	\centering
	\includegraphics[width=1\textwidth]{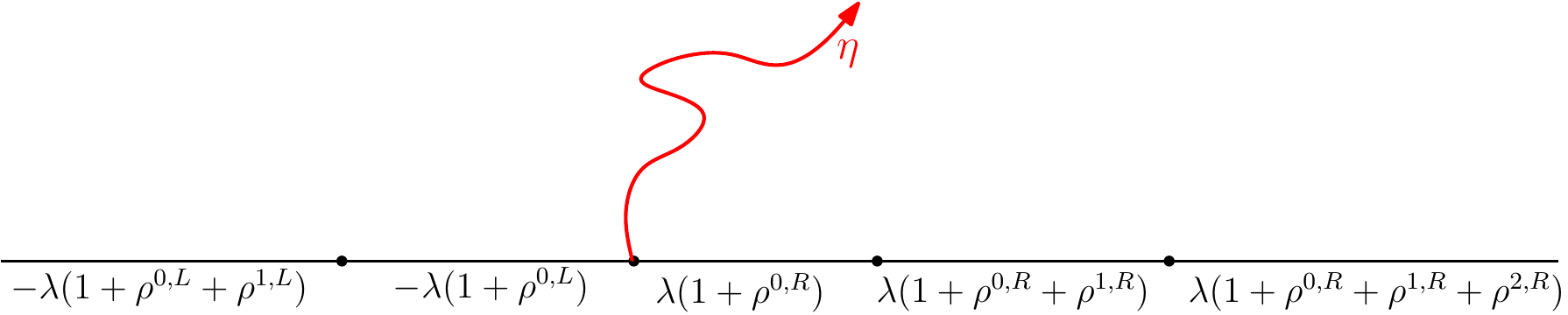}
	\caption{An $\SLE_\kappa(\rho^{0,L},\rho^{1,L};\rho^{0,R},\rho^{1,R},\rho^{2,R})$ process coupled with the GFF $h$ with illustrated boundary conditions as the (zero angle) flow line of $h$. The $\theta$ angle flow line of $h$ then has the law as $\SLE_\kappa(\rho^{0,L}-\frac{\theta\chi}{\lambda},\rho^{1,L};\rho^{0,R}+\frac{\theta\chi}{\lambda},\rho^{1,R},\rho^{2,R})$ process.}
	\label{fig:igflowline}
\end{figure}

So far we have discussed $\SLE_\kappa(\underline{\rho})$ processes on the upper half plane, and for general simply connected domains, the definition can be extended via conformal mappings. Namely,  let $x,y\in\partial D$, $\underline{\tilde{x}}\subset\partial D$ be the force points and $\psi:D\to\bbH$ be a conformal map with $\psi(x) = 0, \psi(y)=\infty$. Then a sample from the chordal $\SLE_\kappa(\underline{\rho})$ process in $D$ from $x$ to $y$ is obtained by first taking an curve $\eta$ from $\SLE_\kappa(\underline{\rho})$ with force points $\psi(\underline{\tilde{x}})$ and then output $\psi^{-1}(\eta)$. {Observe that the term $x^{i,q}\cdot(\psi_\eta^{i,q})'(x^{i,q})$ in~\eqref{eq:def-sle-weight-confradius} is invariant under dilations of $\bbH$, which implies that for $a>0$ and an $\wt{\SLE}_\kappa(\underline{\rho};\underline{\alpha})$ process $\eta$ with force points $\underline{x}$, the law of $\psi\circ\eta$ is $\wt{\SLE}_\kappa(\underline{\rho};\underline{\alpha})$ with force points $a\underline{x}$. This implies that the} notion of $\wt{\SLE}_\kappa(\underline{\rho};\underline{\alpha})$ can also be extended to general simply connected domains by the same way. Moreover, if $\eta$ is a flow line of some GFF $h$, then $\psi^{-1}(\eta)$ is the flow line of $h\circ \psi-\chi\arg \psi'$ in $D$ from $\psi^{-1}(0)$ to $\psi^{-1}(\infty)$. 

{To simplify our language, we are going to extend the notion of $\wt{\SLE}_\kappa(\underline{\rho};\underline{\alpha})$ processes {to} certain non-simply connected domains. %Let ${\rho}^{k+1,L} = -\sum_{i=0}^k\rho^{i,L}$ and ${\rho}^{\ell+1,R} = -\sum_{i=0}^\ell\rho^{i,R}$. 
	Let $D\subset\mathbb{C}$ be some domain and $x,y\in\partial D$, such that the boundary $\partial D$ consist of two non-crossing simple curves $\eta_D^L,\eta_D^R$ running from $x$ to $y$ which possibly intersect and bounce-off each other. Let { $\underline{x}^L:=[(x^{k,L},...,x^{1,L},x^{0,L})]\subset\eta_D^L$ and $\underline{x}^R:=[(x^{0,R},x^{1,R},...,x^{\ell,R})]\subset\eta_D^R$} with $x^{0,L}=x^-$ and $x^{0,R}=x^+$, such that for $i\ge 1$ and $q\in\{L,R\}$, none of the $x^{i,q}$'s lies on $\eta_D^L\cap\eta_D^R$. Further assume $\eta_D^L$ visits $\underline{x}^L$ in the order of $x^{0,L},...,x^{k,L}$, and $\eta_D^R$ visits $\underline{x}^R$ in the order of $x^{0,R},...,x^{\ell,R}$.  On each connected component $\tilde{D}$ of $D$, let $x_{\tilde{D}}$ (resp.\ $y_{\tilde{D}}$) be the first (resp.\ last) point on $\partial \tilde{D}$ traced by $\eta_D^L$, and let $i_{\tilde{D}}^L$ and $j_{{\tilde{D}}}^L$ (resp.\ $i_{\tilde{D}}^R$ and $j_{{\tilde{D}}}^R$) be the largest and smallest integer such that $\partial \tilde{D}\cap\eta_D^L$ (resp.\ $\partial \tilde{D}\cap\eta_D^R$) is between $x^{i_{\tilde{D}},L}$ and $x^{j_{\tilde{D}},L}$ (resp.\ $x^{i_{\tilde{D}},R}$ and $x^{j_{\tilde{D}},R}$). %If $\tilde{D}$ does not contain $x$, then 
	{Let $\mu_{\tilde D}$ be the measure $\wt{\SLE}_\kappa(\sum_{i=0}^{i_{\tilde{D}}^L}\rho^{i,L},\rho^{i_{\tilde{D}}^L+1,L}, ..., \rho^{j_{\tilde{D}}^L-1,L}; \sum_{i=0}^{i_{\tilde{D}}^R}\rho^{i,R},\rho^{i_{\tilde{D}}^R+1,R}, ..., \rho^{j_{\tilde{D}}^R-1,R};\alpha^{i_{\tilde{D}}^L+1,L}, ..., \alpha^{j_{\tilde{D}}^L-1,L};\\\alpha^{i_{\tilde{D}}^R+1,R}, ..., \alpha^{j_{\tilde{D}}^R-1,R})$ in $\eta_{\tilde{D}}$ for curves running from $x_{\tilde{D}}$ to $y_{\tilde{D}}$ with force points $x_{\tilde{D}}^-,x^{i_{\tilde{D}}^L+1,L},\\ ..., x^{j_{\tilde{D}}^L-1,L};x_{\tilde{D}}^+,x^{i_{\tilde{D}}^R+1,R}, ..., x^{j_{\tilde{D}}^R-1,R}$. Sample $(\eta_{\tilde D})_{\tilde D}$ from the product measure $\prod_{\tilde D}\mu_{\tilde D}$.} Concatenate all the $\eta_{\tilde{D}}$'s, and define the law of this curve from $x$ to $y$ in by $\wt{\SLE}_\kappa(\underline{\rho};\underline{\alpha})$ in $D$ with force points $\underline{x}$.}
%If   $\tilde{D}$  contains $x$, then the corresponding curve $\eta_{\tilde{D}}$ from $x$ to $y_{\tilde{D}}$ is sampled from $\wt{\SLE}_\kappa(\rho^{0,L},\rho^{1,L}, ..., \rho^{j_{\tilde{D}}^L-1,L}; \rho^{0,R},\rho^{i_{\tilde{D}}^R+1,R}, ..., \rho^{j_{\tilde{D}}^R-1,R};\alpha^{i_{\tilde{D}}^L+1,L}, ..., \alpha^{j_{\tilde{D}}^L-1,L};\alpha^{i_{\tilde{D}}^R+1,R}, ..., \alpha^{j_{\tilde{D}}^R-1,R})$

{We remark that our definition above is natural in the following sense. Temporarily assume $\underline{\alpha}$ is zero. Let $\psi^L:\mathbb{C}\backslash\eta_D^L\to\mathbb{C}\backslash (-\infty,0)$ and $\psi^R:\mathbb{C}\backslash\eta_D^R\to\mathbb{C}\backslash (0,\infty)$ be the  conformal maps sending $x$ to 0 and $y$ to $\infty$. Let $V_0^{i,L} = \psi^L(x^{i,L})$ and $V_0^{i,R}= \psi^R(x^{i,R})$. Consider a GFF $h$ on $D$ with boundary conditions such that $h\circ(\psi^L)^{-1}-\chi\arg((\psi^{L})^{-1})'$ agrees with~\eqref{eq:flowlinecouple} on $(-\infty,0)$ and  $h\circ(\psi^R)^{-1}-\chi\arg((\psi^{R})^{-1})'$ agrees with ~\eqref{eq:flowlinecouple} on $(0,\infty)$ {with $t=0$}. In each connected component $\tilde{D}$ construct the flow line $\eta_{\tilde{D}}$ of $h$ from $x_{\tilde{D}}$ to $y_{\tilde{D}}$, and the $\SLE_\kappa(\underline{\rho
	})$ process in $D$ can be understood as the concatenation of all the $\eta_{\tilde{D}}$'s. For non-zero $\underline{\alpha}$ we can further weight by the corresponding conformal derivatives.   } 

The $\SLE_\kappa(\underline{\rho})$ curve $\eta$ satisfies the following Domain Markov property. Let $\tau$ be some stopping time for $\eta$. On the event that $\tau$ is less than the continuation threshold, the conditional law of $\eta(t+\tau)_{t\ge0}$ given $\eta([0,\tau])$ is an $\SLE_\kappa(\underline{\rho})$ on $\bbH\backslash K_\tau$ with force points $\underline{x_\tau}$, where $x_\tau^{i,L} = \inf\{ x:x\in\{x^{i,L}\cup (\eta([0,\tau])\cap\bbR)\}\}$ and $x_\tau^{i,R} = \sup\{ x:x\in\{x^{i,R}\cup (\eta([0,\tau])\cap\bbR)\}\}$, and if two force points $x^{i,q}$ and $x^{j,q}$ are equal, they could be merged into a single force point of weight $\rho^{i,q}+\rho^{j,q}$. 

% Finally we recall the following boundary hitting regime for $\SLE_\kappa(\underline{\rho})$.
% \begin{lemma}\label{lem:SLE-bdryhitting}
	% Fix $\kappa>0$. Let $\eta$ be an $\SLE_\kappa(\underline{\rho})$ curve in $\bbH$ with force points $\underline{x}$ such that $\underline{\rho}$ satisfies~\eqref{eq:continuation-threshold} (i.e., the continuation threshold is never hit). Then:
	
	% (i) If $\rho^{0,R}+...+\rho^{j,R}\ge \frac{\kappa}{2}-2$ for some $0\le j\le \ell$, then $\eta$ almost surely does not hit $(x^{j,R},x^{j+1,R})$;
	
	% (ii) Let $(I_j)_{1\le j\le m}$ be a collection of the intervals of the form $(x^{i+1,L},x^{i,L})$ or $(x^{i,R},x^{i+1,R})$ such that $\max\{-2,\frac{\kappa}{2}-4\}<\rho^{0,L}+...+\rho^{i,L}< \frac{\kappa}{2}-2$ or $\max\{-2,\frac{\kappa}{2}-4\}<\rho^{0,R}+...+\rho^{i,R}< \frac{\kappa}{2}-2$. Then with positive probability, $\eta$ hits and bounces off over all of the $I_j$'s. 
	% \end{lemma}
% \begin{proof}
	% (i) is stated in~\cite[Remark 5.3]{MS16a}. For (ii), when $m=1$, the statement follows from~\cite[Lemma 2.1]{MW17}, and for general $m>0$, the claim follows directly from an induction and the Domain Markov property of $\SLE_\kappa(\underline{\rho})$. 
	% \end{proof}

\subsection{The coupling of the two flow lines}\label{subsec:pre-resample}
One important implication of the flow line coupling of SLE and the GFF is that, for two $\SLE_\kappa(\underline{\rho})$ processes $\eta_1$ and $\eta_2$ coupled within the same imaginary geometry, one can easily read off the conditional laws of $\eta_1$ given $\eta_2$ and $\eta_2$ given $\eta_1$. Suppose $\eta_1$ and $\eta_2$ are flow lines of $h$, then given $\eta_1$, the conditional law of $\eta_2$ is the same as the law of the flow line (with some angle) of the GFF in $\mathbb{H}\backslash\eta_1$ with the \textit{flow line boundary conditions} (see \cite[Figure 1.10]{MS16a} for more explanation) induced by $\eta_1$, and vice versa for the law of $\eta_1$ given $\eta_2$.

Now we state the following commutation relation between $\wt{\SLE}_\kappa(\underline{\rho};\underline{\alpha})$ processes. See Figure~\ref{fig:commutation} for an illustration. {Suppose $(\Omega,\mathcal{F})$ is a $\sigma$-finite measure space and $X:\Omega\to A$ is a random variable with law $\mu$. Also suppose $(\nu_x)_{x\in A}$ is a family of $\sigma$-finite measures on $(\Omega,\mathcal{F})$. By first sampling $X$ from $\mu$ and then $Y$ from $\nu_X$, we refer to a sample $(X,Y)$ from the measure $\nu_x(dy)\mu(dx)$ on $(\Omega,\mathcal{F})$.}
\begin{proposition}\label{prop:commutation}
	Let $\kappa\in(0,4)$. Fix $x^{k,L}<...<x^{1,L}<x^{0,L}= 0^-<x^{0,R}= 0^+< x^{1,R}<...<x^{\ell, R}$, $\rho^{i,q}\in\bbR$, $\alpha^{i,q}\in\bbR$ for $q\in\{L,R\}$ and $\rho>-2$. Let $\tilde{\rho} = (\rho+2+\underline{\rho}^L;-\rho-2+\underline{\rho}^R)$.  Suppose that $\underline{\rho},\underline{\tilde{\rho}}$ both satisfy~\eqref{eq:continuation-threshold}.  
	The following three laws on pairs of curves $(\eta_1, \eta_2)$ agree:
	\begin{itemize}
		\item Sample $\eta_1$ in $\bbH$ from $0$ to $\infty$ as $\wt{\SLE}_\kappa(\underline{\rho};\underline{\alpha})$ with force points $\underline{x}$. Then sample an $\wt{\SLE}_\kappa(\rho;\underline{\tilde{\rho}}^R;{0;\underline{\alpha}^R})$ process $\eta_2$ in the right part of $\bbH\backslash\eta_1$  with force points $(0^-;\underline{x}^R)$;
		% On each connected component $D$ of $\bbH\backslash\eta_1$ to the right of $\eta_1$, let $x_D$  (resp.\ $y_D$) be the first (resp.\ last) point on $\partial D$ traced by $\eta$. Let $0\le i_D<j_D\le \ell+1$ be the largest and smallest integer such that   $x^{i_D,R}<x_D<y_D<x^{j_D,R}]$. Sample an $\wt{\SLE}_\kappa(\rho-2;\sum_{i=0}^{i_D}\rho^{i,R},\rho^{i_D+1,R},...,\rho^{j_D-1,R};\alpha^{i_D+1,R},...,\alpha^{j_D-1,R})$ curve on $D$ running from $x_D$ to $y_D$ with force points $x_D^-$, $x_D^+$, $x^{i_D+1,R},...,x^{j_D-1,R}$ (if $j_D=i_D+1$ then the curve is understood as ${\SLE}_\kappa(\rho-2;\sum_{i=0}^{i_D}\rho^{i,R})$). Let $\eta_2$ be the concatenation of all these curves.
		\item Sample $\eta_2$ in $\bbH$ from $0$ to $\infty$ as $\wt{\SLE}_\kappa(\underline{\tilde{\rho}};\underline{\alpha})$ with force points $\underline{x}$. Then sample an $\wt{\SLE}_\kappa(\underline{\rho}^L;\rho;{\underline{\alpha}^L;0})$ process $\eta_1$ in the left part of $\bbH\backslash\eta_2$  with force points $(\underline{x}^L;0^+)$;
		% On each connected component $D$ of $\bbH\backslash\eta_2$ to the left of $\eta_1$, let $x_D$  (resp.\ $y_D$) be the first (resp.\ last) point on $\partial D$ traced by $\eta$. Let $0\le i_D<j_D\le k+1$ be the largest and smallest integer such that   $x^{i_D,L}>x_D>y_D<x^{j_D,L}$. Sample an $\wt{\SLE}_\kappa(\sum_{i=0}^{i_D}\rho^{i,L},\rho^{i_D+1,L},...,\rho^{j_D-1,L};\rho-2;\alpha^{i_D+1,L},...,\alpha^{j_D-1,L})$ curve on $D$ running from $x_D$ to $y_D$ with force points $x_D^-$, $x^{i_D+1,L},...,x^{j_D-1,L}$ $x_D^+$. Let $\eta_1$ be the concatenation of all these curves.
		{\item Sample $\eta_1$ in $\bbH$ from $0$ to $\infty$ as ${\SLE}_\kappa(\underline{\rho})$ with force points $\underline{x}$. Then sample an ${\SLE}_\kappa(\rho;\underline{\tilde{\rho}}^R)$ process in the right part of $\bbH\backslash\eta_1$ with force points $(0^-;\underline{x}^R)$. {For $i\ge 1$ and $j=1,2$}, let $D_{\eta_j}^{i,q}$ be the connected component of $\bbH\backslash\eta_j$ with $x^{i,q}$ on the boundary, $\sigma_{\eta_j}^{i,q}$ (resp.\ $\xi_{\eta_j}^{i,q}$) be the first (resp.\ last) point on $\partial D_{\eta_j}^{i,q}$ traced by $\eta_j$, and $\psi_{\eta_j}^{i,q}:D_{\eta_j}^{i,q}\to\bbH$ be the conformal map sending $(\sigma_{\eta_j}^{i,q},x^{i,q},\xi_{\eta_j}^{i,q})$ to $(0,\pm 1,\infty)$ where we{take the + sign when $q=R$ and the - sign when $q=L$}. Now weight the law of $(\eta_1,\eta_2)$ by 
			$$\prod_{i=1}^k\left|x^{i,L}\cdot(\psi_{\eta_1}^{i,L})'(x^{i,L})\right|^{\alpha^{i,L}}\cdot \prod_{i=1}^\ell \left|x^{i,R}\cdot(\psi_{\eta_2}^{i,R})'(x^{i,R})\right|^{\alpha^{i,R}}.$$
		}
	\end{itemize}
\end{proposition}

\begin{figure}
	\centering
	\includegraphics[width=0.95\textwidth]{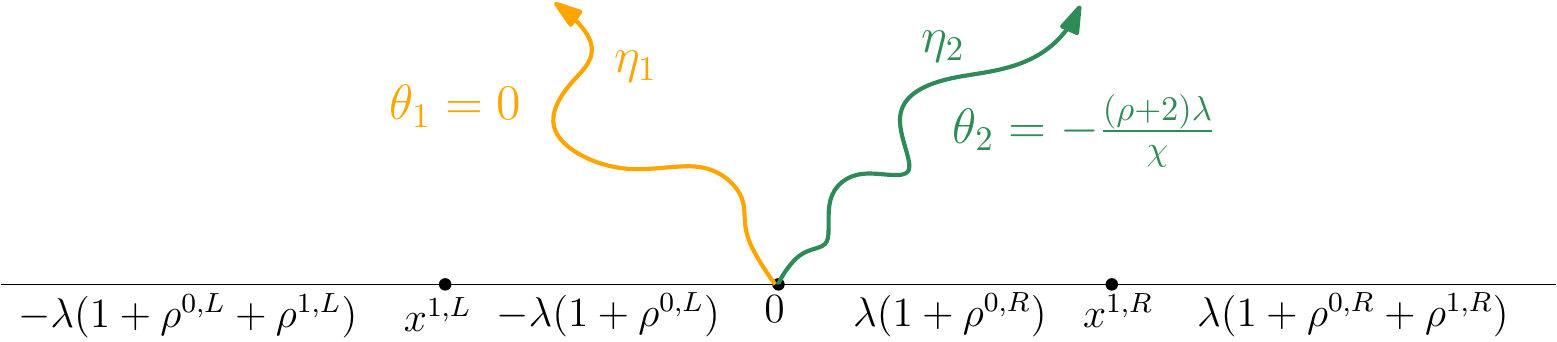}
	\caption{Let $h$ be the GFF with illustrated boundary conditions and $(\eta_1,\eta_2)$ be the angle $(0,-\frac{(\rho+2)\lambda}{\chi})$ flow lines of $h$. Then conditioned on $\eta_1$, $\eta_2$ is an $\SLE_\kappa(\rho;-\rho-2+\rho^{0,R},\rho^{1,R})$ process, and conditioned on $\eta_2$, $\eta_1$ is an $\SLE_\kappa(\rho^{0,L},\rho^{1,L};\rho)$ process. By Lemma~\ref{lem:resampling}, these two conditional laws uniquely characterize the joint law $(\eta_1,\eta_2)$. In Proposition~\ref{prop:commutation}, we  prove that once we weight the law of $(\eta_1,\eta_2)$ by $|x^{1,L}(\psi_{\eta_1}^{1,L})'(x^{1,L})|^{\alpha^{1,L}}\cdot |x^{1,R}(\psi_{\eta_2}^{1,R})'(x^{1,R})|^{\alpha^{1,R}}$, a sample of $(\eta_1,\eta_2)$ can be produced by first sampling an $\wt{\SLE}_\kappa(\rho^{0,L},\rho^{1,L};\rho^{0,R},\rho^{1,R};\alpha^{1,L};\alpha^{1,R})$ process $\eta_1$ and then an $\wt{\SLE}_\kappa(\rho;-\rho-2+\rho^{0,R},\rho^{1,R};0;\alpha^{1,R})$ curve $\eta_2$ in the right part of $\bbH\backslash\eta_1$ (which possibly induces a weighting on the law of $\eta_1$), or equivalently first sampling an $\wt{\SLE}_\kappa(\rho+2+\rho^{0,L},\rho^{1,L};-\rho-2+\rho^{0,R},\rho^{1,R};\alpha^{1,L};\alpha^{1,R})$ process $\eta_2$ and then an $\wt{\SLE}_\kappa(\rho^{0,L},\rho^{1,L};\rho;\alpha^{1,L};0)$ process $\eta_1$ in the left part of $\bbH\backslash\eta_2$.}
	\label{fig:commutation}
\end{figure}

{We remark that the topological configuration of the two curves $(\eta_1,\eta_2)$ could be rather complicated, as they may intersect other, and both intersect the boundaries $(-\infty,0)$ and $(0,\infty)$, and we shall apply the definition of $\wt{\SLE}_\kappa(\underline{\rho};\underline{\alpha})$ processes in non-simply connected domains as specified in the previous section.}
% For simplicity, we shall denote the above conditional law of $\eta_2$  given $\eta_1$ by $\wt{\SLE}_\kappa(\rho-2;\underline{\rho}^R;\underline{\alpha}^R)$  in $\bbH\backslash\eta_1$ to the right of $\eta_1$, and the conditional law of $\eta_1$ given $\eta_2$ by $\wt{\SLE}_\kappa(\underline{\rho}^L;\rho-2;\underline{\alpha}^L)$ in $\bbH\backslash\eta_2$ to the left of $\eta_2$.
\begin{proof}
	If $\alpha^{i,q}=0$ for all possible $i$ and $q$, then as argued in~\cite[Section 6]{MS16a}, the pairs $(\eta_1,\eta_2)$ generated from the three ways can all be realized as  sampling the angle $(\theta_1,\theta_2)=(0,-\frac{(\rho+2)\lambda}{\chi})$ flow lines of the GFF with boundary conditions as~\eqref{eq:flowlinecouple} and~\eqref{eq:ig-gff}  and therefore the claim follows. Let $\mathcal{P}$ be the corresponding law of these two flow lines, which agrees with the law of $(\eta_1,\eta_2)$ constructed as in the third way before we do the weighting.
	
	Now for general $\alpha^{i,q}$ we let $\mathcal{L}$ be the law on the pairs $(\eta_1,\eta_2)$ constructed as in the first way of the statement (i.e., first sample $\eta_1$ and then $\eta_2$). For each $1\le i\le\ell$, let $j_{\eta_1}^{i,R}$ be the smallest $j\ge 1$ such that $x^{j,R}\in D_{\eta_1}^{i,R}$, and let $y_{\eta_1}^{i,R} = \psi_{\eta_1}^{j_{\eta_1}^{i,R},R}(x^{i,R})$. Then by definition of the $\wt{\SLE}_\kappa(\underline{\rho};\underline{\alpha})$ processes, we have
	\begin{equation}\label{eq:prop-commutation-1}
		\frac{d\mathcal{L}}{d\mathcal{P}}(\eta_1,\eta_2) = \left(\prod_{q\in\{L,R\}}\prod_{i\ge1}|x^{i,q}\cdot(\psi_{\eta_1}^{i,q})'(x^{i,q})|^{\alpha^{i,q}}\right)\cdot \prod_{i=1}^\ell \left|y_{\eta_1}^{i,R}\cdot(\psi_{\eta_2|\eta_1}^{i,R})'(y_{\eta_1}^{i,R})\right|^{\alpha^{i,R}}
	\end{equation}
	where $\psi_{\eta_2|\eta_1}^{i,R}$ is the corresponding conformal map $\psi_{\eta_1}^{j_{\eta_1}^{i,R},R}(D_{\eta_2}^{i,R})\to\bbH$ with $(\psi_{\eta_1}^{j_{\eta_1}^{i,R},R}(\sigma_{\eta_2}^{i,R}),\\y_{\eta_1}^{i,R},   \psi_{\eta_1}^{j_{\eta_1}^{i,R},R}(\xi_{\eta_2}^{i,R}))$  mapped to $(0,1,\infty)$. Now we observe that $\psi_{\eta_2}^{i,R} = \psi_{\eta_2|\eta_1}^{i,R}\circ\psi_{\eta_1}^{j_{\eta_1}^{i,R},R}$, and by definition we have $$\psi_{\eta_1}^{i,R}(z) = \frac{\psi_{\eta_1}^{j_{\eta_1}^{i,R},R}(z)}{\psi_{\eta_1}^{j_{\eta_1}^{i,R},R}(x^{i,R})}=\frac{\psi_{\eta_1}^{j_{\eta_1}^{i,R},R}(z)}{y_{\eta_1}^{i,R}}$$
	which implies that 
	$$(\psi_{\eta_2}^{i,R})'(x^{i,R}) = (\psi_{\eta_2|\eta_1}^{i,R})'(\psi_{\eta_1}^{j_{\eta_1}^{i,R},R}(x^{i,R}))\cdot (\psi_{\eta_1}^{j_{\eta_1}^{i,R},R})'(x^{i,R}) = (\psi_{\eta_2|\eta_1}^{i,R})'(y_{\eta_1}^{i,R})\cdot y_{\eta_1}^{i,R}\cdot (\psi_{\eta_1}^{i,R})'(x^{i,R})$$
	and therefore~\eqref{eq:prop-commutation-1} can be rewritten as
	\begin{equation}\label{eq:prop-commutation-2}
		\frac{d\mathcal{L}}{d\mathcal{P}}(\eta_1,\eta_2) = \prod_{i=1}^k\left|x^{i,L}\cdot(\psi_{\eta_1}^{i,L})'(x^{i,L})\right|^{\alpha^{i,L}}\cdot \prod_{i=1}^\ell \left|x^{i,R}\cdot(\psi_{\eta_2}^{i,R})'(x^{i,R})\right|^{\alpha^{i,R}}.
	\end{equation}
	Using a similar argument, one can show that if we let $\tilde{\mathcal{L}}$ be    the law on the pairs $(\eta_1,\eta_2)$ constructed as in the second way of the statement, then $\frac{d\tilde{\mathcal{L}}}{d\mathcal{P}}$ is also the same as~\eqref{eq:prop-commutation-2}. Therefore the claim follows.
\end{proof}

Proposition~\ref{prop:commutation} gives three equivalent ways to characterize the joint law of $(\eta_1,\eta_2)$. {On the other hand,}  at least when $\alpha^{i,q}=0$, the two conditional laws $\eta_1|\eta_2$ and $\eta_2|\eta_1$ as in Proposition~\ref{prop:commutation} uniquely determines the joint law of $(\eta_1,\eta_2)$.
\begin{lemma}\label{lem:resampling}
	Let $\kappa\in(0,4)$, $\underline{\rho},\rho,\underline{x}$ be the same as Proposition~\ref{prop:commutation}. Suppose $(\eta_1,\eta_2)$ are random non-crossing curves in $\bbH$ from 0 to $\infty$ sampled from some probability measure, such that conditioned on $\eta_1$, $\eta_2$ is an ${\SLE}_\kappa({\rho;\underline{\tilde\rho}^R})$  in the right part of $\bbH\backslash\eta_1$, and conditioned on $\eta_2$, $\eta_1$ is an ${\SLE}_\kappa(\underline{\rho}^L;{\rho})$ in the left part of $\bbH\backslash\eta_2$. Then the joint law of $(\eta_1,\eta_2)$ is the same as in Proposition~\ref{prop:commutation} with $\alpha^{i,q}=0$ for all $i,q$.
\end{lemma}
\begin{proof}
	{When $\eta_1$ a.s.\ does not intersect $\eta_2$ (i.e., $\rho\ge\frac{\kappa}{2}-2$), the claim follows from the same argument as in~\cite[Section 4]{MS16b}. For the remaining case, we may first separate the starting and ending points of $(\eta_1,\eta_2)$ as in the first step of~\cite[Proof of Theorem 4.1]{MS16b} and then apply the same argument in~\cite[Appendix A]{MSW19}. See also Appendix~\ref{sec:appendix} for an alternative proof based on Markov chain irreducibility results in~\cite{Meyn-Tweedie} and Lemma~\ref{lem:sle>0prob}. }
\end{proof}

We are going to use the following variant of  Lemma~\ref{lem:resampling}, which follows from  exactly the same Markov chain remixing argument in~\cite[Theorem 4.1]{MS16b} and~\cite[Appendix A]{MSW19}.

\begin{lemma}\label{lem:resampling-a}
	Let $\kappa\in(0,4)$, $\underline{\rho},\rho,\underline{x}$ be the same as Proposition~\ref{prop:commutation}. For $\e>0$, let $D_{\e} = \cup_{q\in\{L,R\}}\cup_{i\ge1}B(x^{i,q},\e)$. Fix $\e$ sufficiently small such that $0\notin D_\e $. Let $\mathcal{P}$ be the joint law of $(\eta_1,\eta_2)$ as described in Lemma~\ref{lem:resampling}, and $\mathcal{P}_{\e }$ be the probability measure given by conditioning $\mathcal{P}$ on the event $E_\e:=\{\eta_1\cap \overline{D}_{\e}=\eta_2\cap \overline{D}_{\e}=\emptyset\}$. Now suppose {$(\tilde \eta_1,\tilde \eta_2)$} is a sample from some probability measure on curves in $\bbH$ running from 0 to $\infty$, such that the conditional law of $\tilde \eta_2$ given $\tilde \eta_1$ is the ${\SLE}_\kappa({\rho;\underline{\tilde\rho}^R})$  in the right part of $\bbH\backslash \tilde \eta_1$ conditioned on not hitting $\overline{D}_{\e}$,
	and the conditional law of $\tilde \eta_1$ given $\tilde \eta_2$ is the ${\SLE}_\kappa(\underline{\rho}^L;\rho)$ in the left part of $\bbH\backslash \tilde\eta_2$ conditioned on not hitting $\overline{D}_{\e}$. Then the joint law of {$(\tilde \eta_1,\tilde \eta_2)$} is the same as $\mathcal{P}_{\e }$ defined above. 
\end{lemma}

\section{Proof of Theorem~\ref{thm:main}}\label{sec:proof}
In this section, we prove our main result Theorem~\ref{thm:main}. We start with the $\kappa\in (0,4)$ case, where we first extend Theorem~\hyperref[thm:dapeng]{B} to $\SLE_\kappa(\rho;\underline{\rho}^R)$ curves (i.e.\ adding a force point at $0^-$ in Theorem~\hyperref[thm:dapeng]{B}) and then apply the SLE resampling properties  (Lemma~\ref{lem:resampling}). For $\kappa\in (4,8)$ case, we shall use the \emph{SLE duality}, and the $\kappa=4$ case is covered in~\cite[Theorem 1.1.6]{wang2017level}. 

We begin with the following variant of~\cite[Lemma 3.9]{MS17}, which roughly states that flow lines of the GFF can stay arbitrarily close to a given curve. Let $\underline{x},\underline{\rho},\eta$ be as in Theorem~\ref{thm:main}. Recall from~\cite[Remark 5.3]{MS16a} and~\cite[Lemma 2.1]{MW17} that, for fixed $0\le j\le\ell$, if $\rho^{0,R}+...+\rho^{j,R}\ge\frac{\kappa}{2}-2$, then $\eta$ a.s.\ does not hit $(x^{j,R},x^{j+1,R})$, while if $\rho^{0,R}+...+\rho^{j,R}<\frac{\kappa}{2}-2$, then $\eta$ has positive probability of hitting $(x^{j,R},x^{j+1,R})$. Let $J_{\underline{\rho}}$ be the collection of $0\le j\le \ell$ such that $\rho^{0,R}+...+\rho^{j,R}<\frac{\kappa}{2}-2$. For a simple curve $\gamma$ in $\overline{\bbH}$ from 0 to $\infty$ with $\gamma\cap\{x^{1,R},...,x^{\ell,R}\}=\emptyset$, let $J_\gamma=\{0\le j\le\ell:\gamma\cap (x^{j,R},x^{j+1,R})\neq\emptyset\}$. We say that $\gamma$ is admissible w.r.t.\ $(\underline{\rho}^R,\underline{x}^R)$ if $J_\gamma\subset J_{\underline{\rho}}$. In other words, $\gamma$ is admissible if it does not hit the intervals $(x^{j,R},x^{j+1,R})$ where the $\SLE_{\kappa}(\underline{\rho})$ process $\eta$ a.s.\ does not hit.  Similarly, we can define the notion of admissibility for $(\underline{\rho}^L,\underline{x}^L)$. We define the domain $D_\gamma^{i,R}$ to be the connected component of $\bbH\backslash\gamma$ containing $x^{i,R}$, and the points $\sigma_{\gamma}^{j,R}$ and $\xi_{\gamma}^{j,R}$ analogously.  Consider the conformal map $\psi:\bbH\to\bbD$ sending $(0,1,\infty)$ to $(-i,1,i)$.
\begin{lemma}\label{lem:sle>0prob}
	Let $\gamma$ be an admissible curve w.r.t.\ $(\underline{\rho},\underline{x})$ and $\tilde{\gamma} = \psi(\gamma)$. Let $\eta$ be an $\SLE_\kappa(\underline{\rho})$ process in $\bbH$ with force points $\underline{x}$, and $\tilde{\eta}=\psi(\eta)$. For  $\e>0$, define the event $E_\e$ where (i) $\tilde{\eta}$ stays in the $\e$-neighborhood of $\tilde{\gamma}$ and (ii) for any $1\le j\le\ell $, $|\psi(\xi_{\eta}^{j,R})-\psi(\xi_{\gamma}^{j,R})|<\e$  and $|\psi(\sigma_{\eta}^{j,R})-\psi(\sigma_{\gamma}^{j,R})|<\e$. Then for any $\e>0$, the event $E_\e$ happens with positive probability.
\end{lemma}

\begin{figure}
	\centering
	\begin{tabular}{ccc} 
		\includegraphics[scale=0.6]{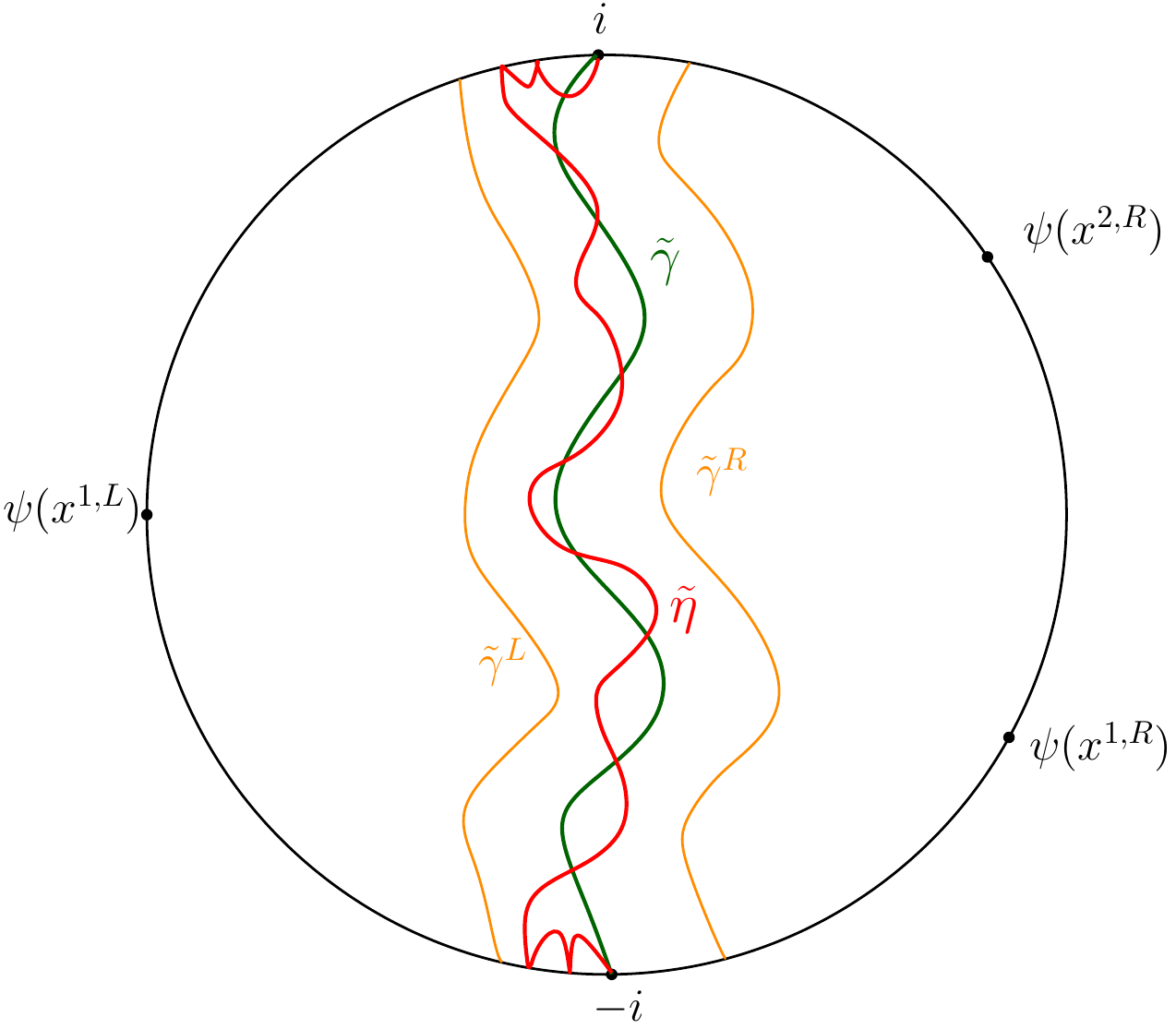}
		&  
		\includegraphics[scale=0.62]{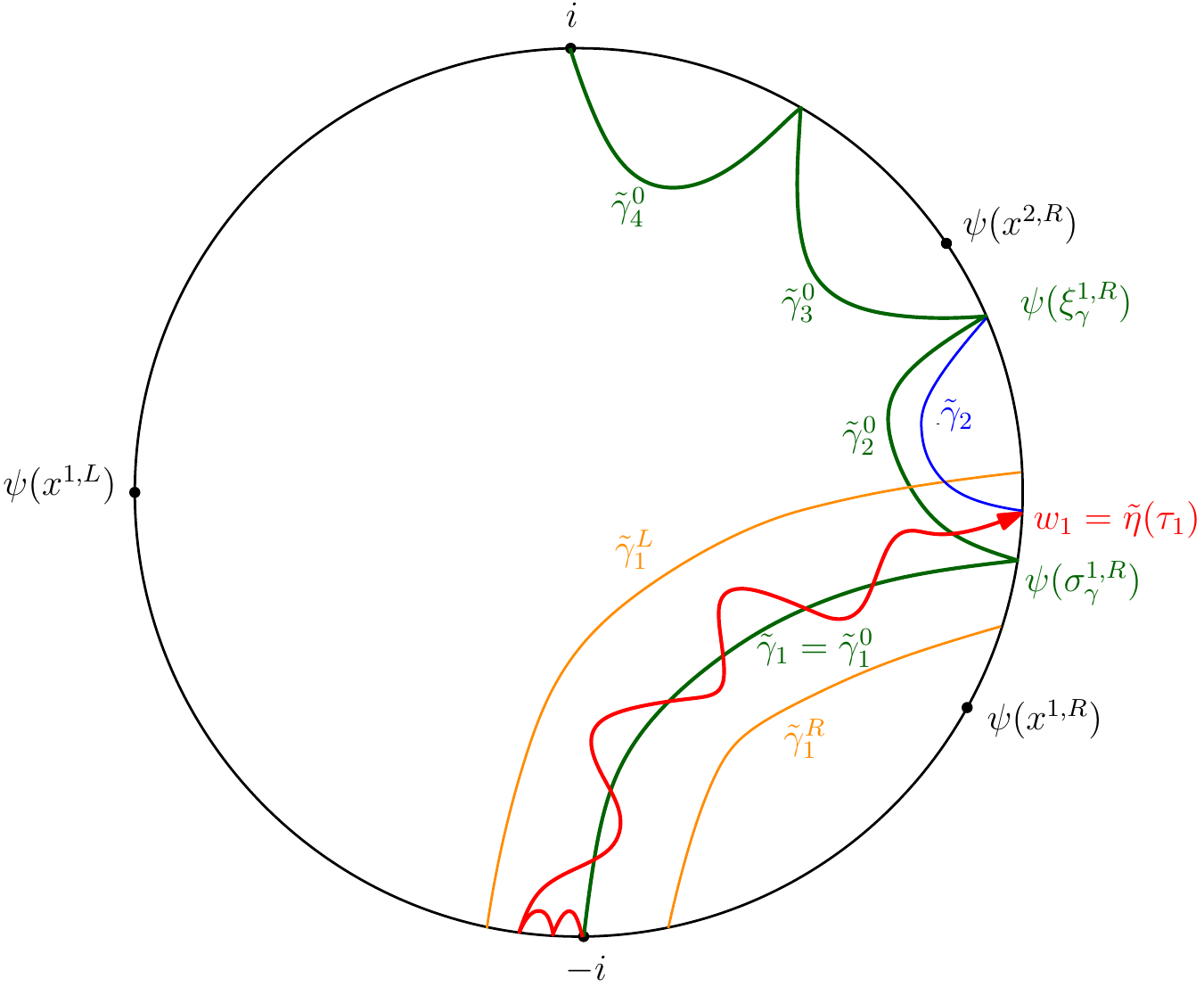}
	\end{tabular}
	\caption{An illustration of Lemma~\ref{lem:sle>0prob}, where we show that flow lines of the GFF can stay arbitraily close to some given curve. \textbf{Left:} the curve $\tilde{\gamma}$ intersects $\partial\bbD$ only at $\pm i$, and we construct the curves $\tilde{\gamma}^L,\tilde{\gamma}^R$ within the $\e$-neighborhood of $\tilde{\gamma}$. Let $U_\gamma$ be the region between $\tilde{\gamma}^L$ and $\tilde{\gamma}^R$, $\tilde{h}_\gamma$ be some GFF in $U_\gamma$ with the same boundary conditions as $\tilde{h}$ on $\partial\bbD\cap\partial U_\gamma$ and flow line boundary conditions on  $\tilde{\gamma}^L,\tilde{\gamma}^R$ such that the flow line $\tilde{\eta}_\gamma$ of $\tilde{h}_\gamma$ from $-i$ to $i$ a.s.\ has positive distance to $\tilde{\gamma}^L\cup\tilde{\gamma}^R$. Then by the GFF absolute continuity argument, the flow line $\tilde{\eta}$ is contained in $U_\gamma$ with positive probability. \textbf{Right}: $\tilde{\gamma}$ consists of 4 arcs in $\bbD$, namely $\tilde{\gamma}_1^0,...,\tilde{\gamma}_4^0$. Let $\tilde{\gamma}_1=\tilde{\gamma}_1^0$ and construct $\tilde{\gamma}_1^L,\tilde{\gamma}_1^R,U_{\tilde{\gamma}_1}$ analogously. Let $I_1$ be the component of $U_{\tilde{\gamma}_1}\cap\bbD$ with $\psi(\sigma_\gamma^{1,R})$ on the boundary.  By the same argument, the event where $\tilde{\eta}$ hits $I_1$ at $w_1$ at some time $\tau_1$ without hitting $\tilde{\gamma}_1^L\cup\tilde{\gamma}_1^R$ has nonzero probability. Then conditioned on $\tilde{\eta}([0,\tau_1])$, we construct the curve $\tilde{\gamma}_2$ from $w_1$ to $\psi(\xi_\gamma^{1,R})$ staying close to $\tilde{\gamma}_2^0$ and iterate the same argument. }\label{fig:posprob}
\end{figure}

\begin{proof}
	We first comment that for each admissible $\gamma$, we may construct some admissible curve $\gamma_\e$ such that (i) $\sigma_{\gamma_\e}^{j,R}=\sigma_{\gamma}^{j,R}$ and $\xi_{\gamma_\e}^{j,R}=\xi_{\gamma}^{j,R}$ for all $1\le j\le \ell$ (ii) $\psi(\gamma_\e)$ is contained in the $\e$-neighborhood of $\tilde{\gamma}$ and (iii) $\gamma_\e\cap(\bbR\cup\{\infty\}) = \{0,\xi_{\gamma_\e}^{1,R},\sigma_{\gamma_\e}^{1,R},...,\xi_{\gamma_\e}^{\ell,R},\sigma_{\gamma_\e}^{\ell,R}\}$. From this point of view, without loss of generality we may assume that $\gamma\cap(\bbR\cup\{\infty\}) = \{0,\xi_{\gamma}^{1,R},\sigma_{\gamma}^{1,R},...,\xi_{\gamma}^{\ell,R},\sigma_{\gamma}^{\ell,R}\}$.
	
	Let $h$ be a GFF on $\bbH$ with boundary conditions~\eqref{eq:ig-gff} such that $\eta$ is the flow line of $h$, and $\tilde{h} = h\circ\psi^{-1}-\chi\arg(\psi^{-1})'$. Then $\tilde{\eta}$ is the flow line of $\tilde{h}$. Assume {$\e$} is sufficiently small such that for $1\le j \le \ell$, $\psi(x^{j,R})$ is not in the $\e$-neighborhood of $\tilde{\gamma}$. We begin with the case where $\gamma\cap\bbR = \{0\}$. We choose some simple path $\tilde{\gamma}^L$ (resp.\ $\tilde{\gamma}^R$) in $\overline{\bbD}\backslash\tilde{\gamma}$ connecting the points $e^{(3\pi/2-\e/6)i}$ and $e^{(\pi/2+\e/6)i}$ (resp.\ $e^{(3\pi/2+\e/6)i}$ and $e^{(\pi/2-\e/6)i}$) such that $\tilde{\gamma}^L\cup \tilde{\gamma}^R$ is contained in the $\e$-neighborhood of $\gamma$, and let $U_\gamma$ be the component of $\bbD\backslash(\tilde{\gamma}^L\cup \tilde{\gamma}^R)$ between $\tilde{\gamma}^L$ and $\tilde{\gamma}^R$. Then as in the proof of~\cite[Lemma 3.9]{MS17}, we may construct a GFF $\tilde{h}_\gamma$ in $U_\gamma$ with the same boundary conditions on {$\partial \bbD\cap \partial U_\gamma$} as $\tilde{h}$ and flow line boundary conditions (see e.g.~\cite[Figure 1.10]{MS16a}) on $\tilde{\gamma}^L\cup\tilde{\gamma}^R$ such that the flow line of $\tilde{h}_\gamma$ from $-i$ to $i$ {has the law as $\SLE_\kappa(\rho^{0,L},-\rho^{0,L},\sum_{j=0}^k\rho^{j,L};\rho^{0,R},-\rho^{0,R}, \sum_{j=0}^\ell\rho^{j,R})$ with force points $(0^-,e^{(3\pi/2-\e/6)i},e^{(\pi/2+\e/6)i};0^+,e^{(3\pi/2+\e/6)i},e^{(\pi/2-\e/6)i})$  and thus} a.s.\ has positive distance from $\tilde{\gamma}^L\cup \tilde{\gamma}^R$. We may choose some {(non-random)} constant $\zeta>0$ small such that this distance is at least $\zeta$ with probability greater than 1/2. Therefore it follows from the same argument of~\cite[Lemma 3.9]{MS17} that, {if we set $U_\gamma^\zeta:=\{z\in U_\gamma;\dist(z, \tilde{\gamma}^L\cup \tilde{\gamma}^R)>\zeta\}$, by the GFF absolute continuous property~\cite[Proposition 3.4]{MS16a},  the law of $\tilde{h}_\gamma$ is absolutely continuous w.r.t.\ $\tilde{h}$ when restricted to the domain $U_\gamma^\zeta$. Since flow lines are a.s. determined by and local sets of the GFF~\cite[Theorem 1.2]{MS16a}, and the flow line of $\tilde{h}_\gamma$ is contained within $U_\gamma^\zeta$ with probability at least 1/2, it follows that}, with positive probability $\tilde{\eta}$ is contained in $U_\gamma^\zeta$ and thus in the $\e$-neighborhood of $\tilde{\gamma}$, which finishes the case when $\gamma\cap\bbR=\{0\}$.
	
	For rest of the case, we write $\tilde{\gamma} = \cup_{i=1}^m\tilde{\gamma}_i^0$ where each $\tilde{\gamma}_i^0:[0,1]\to\overline{\bbD}$ is a subarc of $\tilde{\gamma}$ intersecting $\partial \bbD$ only at the endpoints, and they are aligned in the order traced by $\tilde{\gamma}$. Let $\tilde{\gamma}_1 = \tilde{\gamma}_1^0$. We construct the simple curve $\tilde{\gamma}_1^L$ (resp.\ $\tilde{\gamma}_1^R$) in $\overline{\bbD}\backslash\tilde{\gamma}_1$ within the $\e/m$ neighborhood of $\tilde{\gamma}_1$ connecting a point on $\partial\bbD$ on the left (resp.\ right) side of $\tilde{\gamma}_1$ and a point on $\partial\bbD$ on the same side of $\tilde{\gamma}_1$. Let $U_{\tilde{\gamma}_1}$ be the component of $\bbD\backslash(\tilde{\gamma}_1^L\cup \tilde{\gamma}_1^R)$ between $\tilde{\gamma}_1^L$ and $\tilde{\gamma}_1^R$, and $I_1$ be the component of $\partial U_{\tilde{\gamma}_1}\cap\partial \bbD$ with $\tilde{\gamma}_1(1)$ on its boundary.  
	Then  as above and~\cite[Lemma 3.9]{MS17}, {we may construct a GFF $\tilde{h}_{\tilde{\gamma}_1}$ on $U_{\tilde{\gamma}_1}$ with same boundary conditions as $\tilde{h}$ on $\partial U_{\tilde{\gamma}_1}\cap\partial \bbD$ and flow line condition on $\tilde{\gamma}_1^L\cup \tilde{\gamma}_1^R$, such that the flow line of $\tilde{h}_{\tilde{\gamma}_1}$ a.s. has positive distance to $\tilde{\gamma}_1^L\cup \tilde{\gamma}_1^R$. Again using the same GFF absolute continuity of $\tilde{h}_{\tilde{\gamma}_1}$ w.r.t. $\tilde{h}$ as above and in~\cite[Lemma 3.9]{MS17},}  the event $E_1$ where the flow line $\tilde{\eta}$ first hits $I_1$ at $w_1$ at time $\tau_1$ without hitting $\tilde{\gamma}_1^L\cup \tilde{\gamma}_1^R$ has positive probability. On $E_1$, we choose a simple curve $\tilde{\gamma}_2:[0,1]\to\overline{\bbD}$ such that $\tilde{\gamma}_2(0)=w_1$, $\tilde{\gamma}_2(1) = \tilde{\gamma}_2^0(1)$, $\tilde{\gamma}_2((0,1))\cap\partial\bbD=\emptyset$ and $\tilde{\gamma}_2$ stays within the $2\e/m$-neighborhood of $\tilde{\gamma}_2^0$. Define $\tilde{\gamma}_2^L,\tilde{\gamma}_2^R, I_2$ analogously (with $\e/m$ replaced by $2\e/m$). It follows from the same argument that, conditioned on $E_1$ and $\tilde{\eta}|_{[0,\tau_1]}$, the event $E_2$ where $\tilde{\eta}$ first hits $I_2$ at $w_2$ at time $\tau_2$ without hitting $\tilde{\gamma}_2^L\cup \tilde{\gamma}_2^R$ has positive probability. Now we can conclude the proof by iterating this process, except that at the final step we construct the curve $\tilde{\gamma}_m$ and apply the argument from the $\gamma\cap\bbR=\{0\}$ case.
	%We run the flow line $\tilde{\eta}$ until some stopping time $\tau_0$ such that $\tilde{\eta}|_{[0,\tau_0]}\subset B(-i,\e/6)$ and $\tilde{\eta}(\tau_0)\in\bbD$. Then we draw a curve $\tilde{\gamma}_\eta^0$ from $\tilde{\eta}(\tau_0)$ to $i$ such that $\tilde{\gamma}_\eta^0\cap\partial \bbD = \{i\}$ and $\tilde{\eta}(\tau_0)$ is contained in the $\e/6$-neighborhood of 
\end{proof}

\begin{lemma}\label{lem:onesided+}
	Theorem~\ref{thm:main} holds for $\kappa\in (0,4)$, $k=0$ and $\rho^{0,L}\le 0$. That is, under $J(z)=-1/z$, the law of time reversal of $\SLE_\kappa(\rho^{0,L};\underline{\rho}^R)$ processes in $\bbH$ with force points $(0^-;\underline{x}^R)$ agrees with the $\frac{1}{Z}\wt{\SLE}_\kappa(\underline{\hat{\rho}};\underline{\hat{\alpha}})$ described in Theorem~\ref{thm:main} for some constant $Z=Z(\underline{\rho},\underline{x}^R)\in (0,\infty)$. % Moreover, the  constant $Z$ does not depend on $\underline{x}^R$.
\end{lemma}

\begin{figure}
	\centering
	\begin{tabular}{ccc} 
		\includegraphics[scale=0.48]{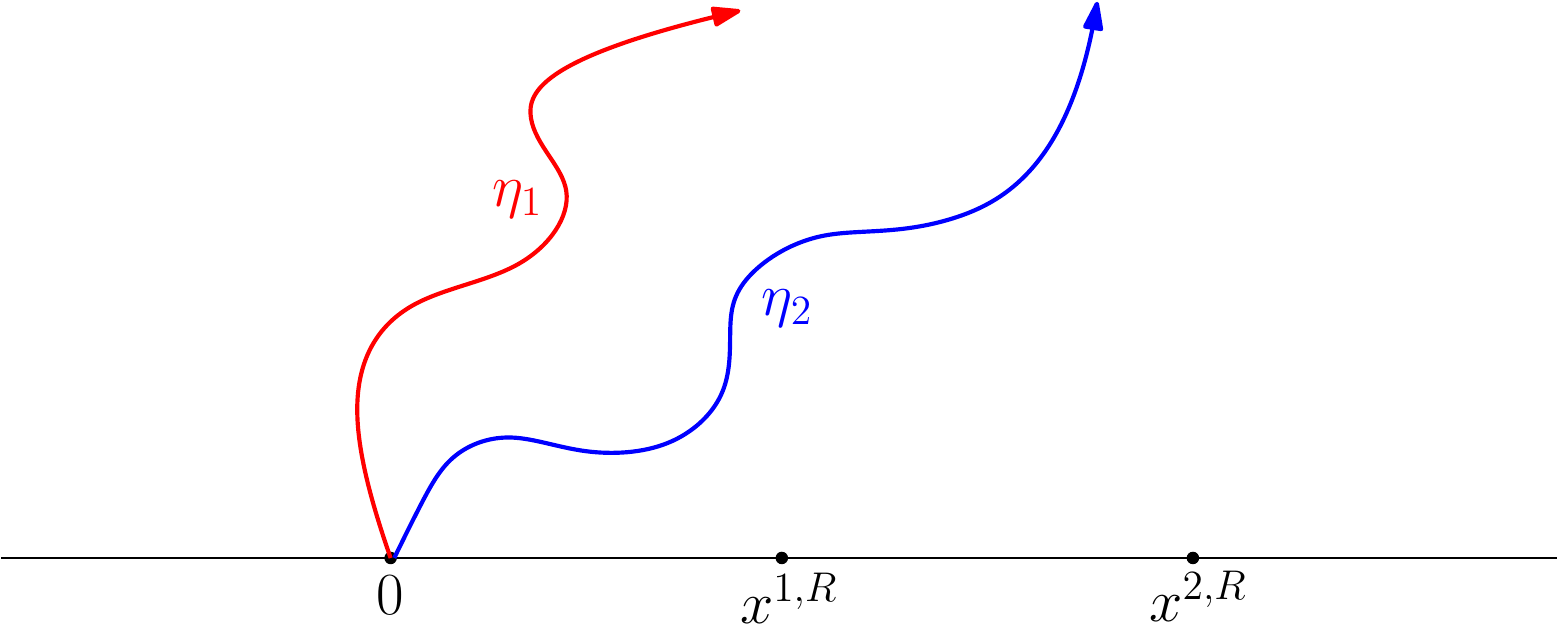}
		& 
		\includegraphics[scale=0.48]{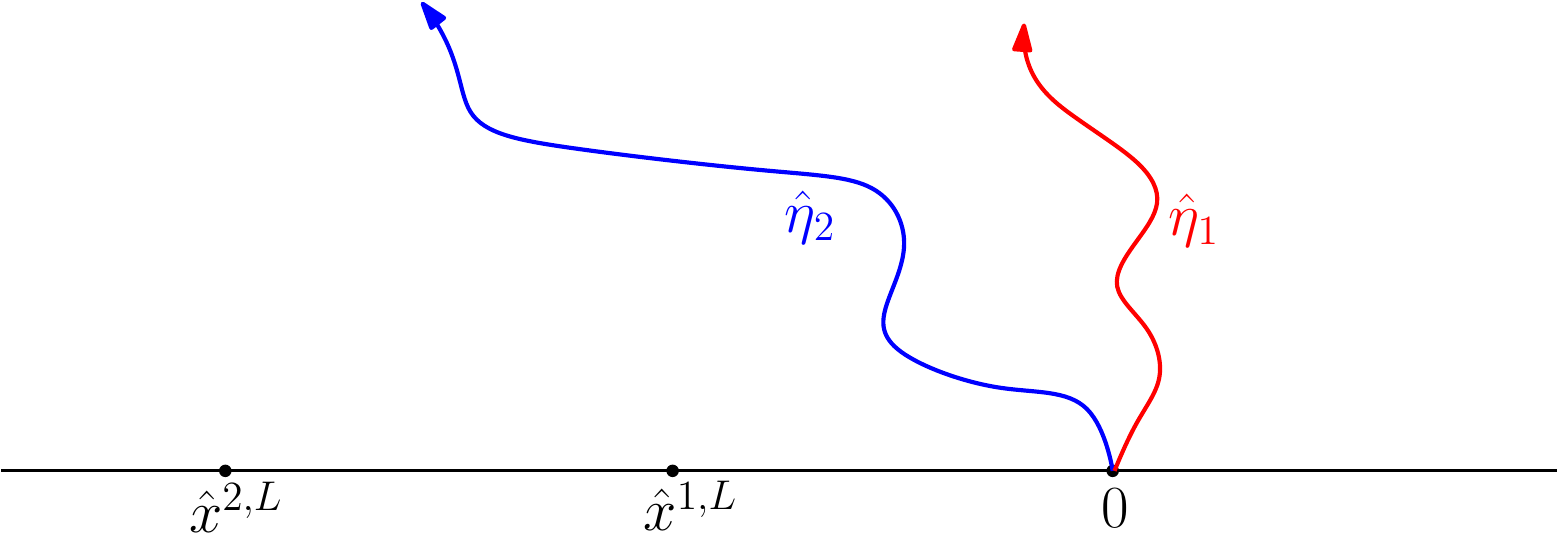}
	\end{tabular}
	\caption{An illustration of the proof of Lemma~\ref{lem:onesided+} for $\ell=2$. \textbf{Left:} We sample an $\SLE_\kappa(\rho^{0,R},\rho^{1,R},\rho^{2,R})$ process $\eta_2$, and an $\SLE_\kappa(-\rho^{0,L}-2;\rho^{0,L})$ process $\eta_1$ in the left part of $\bbH\backslash\eta_2$. Then the conditional law of $\eta_2$ given $\eta_1$ is the $\SLE_\kappa(\rho^{0,L};\rho^{0,R},\rho^{1,R},\rho^{2,R})$ process in the right part of $\bbH\backslash\eta_1$ with force points $0^-;0^+,x^{1,R},x^{2,R}$. \textbf{Right:} For $i=1,2$, $\hat{\eta}_i$ is the time reversal of $J\circ\eta_i$, and $\hat{x}^{1,L}=-1/x^{2,R}$, $\hat{x}^{2,L}=-1/x^{1,R}$. By Theorem~\hyperref[thm:dapeng]{B}, $\hat{\eta}_2$ has the law $\frac{1}{Z}\wt{\SLE}_\kappa(\rho^{0,R}+\rho^{1,R}+\rho^{2,R},-\rho^{2,R},-\rho^{1,R};\frac{\rho^{2,R}(4-\kappa)}{2\kappa},\frac{\rho^{1,R}(4-\kappa)}{2\kappa} )$ with force points $0^-,\hat{x}^{1,L},\hat{x}^{2,L}$. By Theorem~\hyperref[thm:igII]{A}, conditioned on $\hat{\eta}_2$, $\hat{\eta}_1$ is the $\SLE_\kappa(\hat{\rho}^{0,R};-\hat{\rho}^{0,R}-2)$ process  in the right part of $\bbH\backslash\hat{\eta}_2$ with force points $0^-$ and $0^+$. Therefore by Proposition~\ref{prop:commutation}, the conditional law of $\hat{\eta}_2$ given $\hat{\eta}_1$ is $\frac{1}{Z(\eta_1)}\wt{\SLE}_\kappa(\rho^{0,R}+\rho^{1,R}+\rho^{2,R},-\rho^{2,R},-\rho^{1,R};\rho^{0,L};\frac{\rho^{2,R}(4-\kappa)}{2\kappa},\frac{\rho^{1,R}(4-\kappa)}{2\kappa};0 )$ in the left part of $\bbH\backslash\hat{\eta}_1$ with force points $0^-,\hat{x}^{1,L},\hat{x}^{2,L};0^+$, and the claim follows by comparing the two figures. In Lemma~\ref{lem:partition-function}, we further show that the constant $Z(\eta_1)$ is actually independent of $\eta_1$. }\label{fig:onsided}
\end{figure}

%We remark that the final statement on $Z$ is consistent with the result in~\cite{Zhan19}, where in~\cite[(3.16),(3.19)]{Zhan19} and~\cite[Remark 3.6]{Zhan19} the corresponding constant is a hypergeometric function of $\underline{\rho}$ and $\kappa$ and not depending on $\underline{x}$. 

\begin{proof}
	We sample an $\SLE_\kappa(\underline{\rho}^R)$ curve $\eta_2$ in $\bbH$ from 0 to $\infty$ with force points $\underline{x}^R$, and conditioned on $\eta_2$, we sample an $\SLE_\kappa(-\rho^{0,L}-2;\rho^{0,L})$ process $\eta_1$ from 0 to $\infty$ in the left part of $\bbH\backslash\eta_2$ with force points $0^-$ and $0^+$. Then it follows from the construction in Proposition~\ref{prop:commutation} that conditioned on $\eta_1$, $\eta_2$ is an $\SLE_\kappa(\rho^{0,L};\underline{\rho}^R)$ process in the right part of $\bbH\backslash\eta_2$ with force points $(0^-;\underline{x}^R)$. 
	
	%For $0\le i\le \ell$, let $\hat{\rho}^{i,L} = -\rho^{\ell+1-i,R}$ where $\rho^{\ell+1,R} = -\sum_{j=0}^\ell\rho^{j,R}$. 
	For $i=1,2$, let {$\hat{\eta}_i=\mathcal{R}(J\circ\eta_i)$}. Recall the notion of $\hat{x}^{i,q},\hat{\rho}^{i,q},\hat{\alpha}^{i,q}$ in the statement of Theorem~\ref{thm:main}. Now by Theorem~\hyperref[thm:dapeng]{B}, we know that law of $\hat{\eta}_2$ is the probability measure proportional to $\wt{\SLE}_\kappa(\underline{\hat{\rho}}^L;\underline{\hat{\alpha}}^L)$ process from 0 to $\infty$ with force points $\underline{\hat{x}}^L$. Moreover, by Theorem~\hyperref[thm:igII]{A}, the conditional law of $\hat{\eta}_1$ given $\hat{\eta}_2$ is the $\SLE_\kappa(\hat{\rho}^{0,R};-\hat{\rho}^{0,R}-2)$ process in the right part of $\bbH\backslash\hat{\eta}_2$ with force points $0^-$ and $0^+$. 
	Therefore it follows from Proposition~\ref{prop:commutation} that the conditional law of $\hat{\eta}_2$ given $\hat{\eta}_1$ is a constant (possibly depending on $\hat{\eta}_1$) times $\wt{\SLE}_\kappa(\underline{\hat{\rho}}^L;\hat{\rho}^{0,R};\underline{\hat{\alpha}}^L)$ in the left part of $\bbH\backslash\hat{\eta}_1$ with force points $(\underline{\hat{x}}^L;0^+)$. This justifies  the  reversibility of an $\SLE_\kappa(\rho^{0,L};\underline{\rho}^R)$ process in the right part of $\bbH\backslash\eta_1$ with force points $(0^-;\underline{x}^R)$. 
	
	Now consider the conformal map $\psi:\bbH\to\bbD$ sending $(0,1,\infty)$ to $(-i,1,i)$, and let $\tilde{\eta}_j = \psi\circ\eta_j$ for $j=1,2$.  For $\e>0$, let $A_\e$ be the partial annulus $\{z:\frac{\pi}{2}-\e<\arg z<\frac{3\pi}{2}+\e;\ 1-\e<|z|<1\}$. Fix $\e_0>0$ small such that $A_{\e_0}$ contains none of points in $\psi(\underline{x}^R)$ {other than $\psi(x^{0,R})=-i$} and let $\e<\e_0$. Then by Lemma~\ref{lem:sle>0prob},  the event $E_\e'$ where $\tilde{\eta}_1$ is contained in the domain $A_\e$ has positive probability.  On the event $E_\e'$, let $\sigma_\e$ be the last point on the arc $\{z:\frac{3\pi}{2}<\arg z< 2\pi;\ |z|=1\}$ hit by $\tilde \eta_1$, and $\xi_\e$ be the first point on the arc $\{z:0<\arg z< \frac{\pi}{2};\ |z|=1\}$ hit by $\tilde \eta_1$. Let $D_1^\e$ be the connected component of $\bbD\backslash\eta_1$ with 1 on the boundary, and $\psi_\e:D_1^\e\to\bbD$ sending $(\sigma_\e,1,\xi_\e)$ to $(-i,1,i)$. Then $\psi_\e\circ\tilde \eta_2$ is an $\SLE_\kappa(\rho^{0,L};\underline{\rho}^R)$ process in $\bbD$ with force points $((-i)^-;\psi_\e\circ\psi(\underline{x}^R))$ (with $\psi_\e((-i)^+)$ identified as $(-i)^+$), and the law of its time reversal is proportional to the $\wt{\SLE}_\kappa(\underline{\hat{\rho}}^L;\hat{\rho}^{0,R};\underline{\hat{\alpha}}^L;0)$ process. Therefore as we condition on $E_\e'$ and send $\e\to 0$, {$D^\e_1$} converges to $\bbD$ in Caratheodory topology, and $\psi_\e\circ\psi(\underline{x}^R)$ converges to $\psi(\underline{x}^R)$. {To conclude the proof, we {look at the time reversal result of $\psi_\e\circ\eta_2$ in $\bbD$, send $\e\to0$ and apply}  the continuity of $\SLE_\kappa(\underline{\rho})$ processes w.r.t.\ the location of force points from~\cite[Section 2]{MS16a}. To be more precise, suppose $\eta$ and $(\eta^\e)_{\e>0}$ are $\SLE_\kappa(\underline{\rho})$ processes in $\bbD$ from $-i$ to $i$ with force points $\underline{y}$ and $\underline{y}^\e$, such that $y^{0,L,\e}=(-i)^-$, $y^{0,R,\e}=(-i)^+$ and $y^{j,q,\e}\to y^{j,q}$ as $\e\to 0$. Let $h$ and $(h^\e)_{\e>0}$ be the corresponding GFF on $\bbD$ such that $\eta$ and $(\eta^\e)_{\e>0}$ are the flow lines of $h$ and $(h^\e)_{\e>0}$ from $-i$. For $\delta>0$, let $D_\delta':=\cup_{q\in\{L,R\}}\cup_{j\ge1}B(y^{j,L},\delta)$. By~\cite[Proposition 3.4, Remark 3.5]{MS16a}, the total variation distance between $h^\e|_{\bbD\backslash D_\delta'}$ and $h|_{\bbD\backslash D_\delta'}$ goes to 0 as $\e\to 0$ for fixed $\delta$. Since flow lines are deterministic functions {and local sets} of the GFF~\cite[Theorem 1.2]{MS16a}, it follows that the law of $\eta^\e$ conditioned on not hitting $D_\delta'$ converges in total variation distance to that of $\eta$ conditioned on not hitting $D_\delta'$ as $\e\to 0$. From this argument, the law of the time reversal of an $\SLE_\kappa(\rho^{0,L};\underline{\rho}^R)$ process conditioned on having distance $\delta$ to $x^{j,R}$ for $j\ge 1$   agrees with $\wt{\SLE}_\kappa(\underline{\hat{\rho}}^L;\hat{\rho}^{0,R};\underline{\hat{\alpha}}^L;0)$ conditioned on the same event (up to a multiplicative constant), and the claim follows by taking $\delta\to 0$}.
\end{proof}

\begin{corollary}\label{cor}
	Let $\underline{\rho},\underline{x},\underline{\hat{x}},\underline{\hat{\rho}}$ be as in Lemma~\ref{lem:onesided+}, and $\gamma$ be a curve in $\overline{\bbH}$ from 0 to $\infty$ which is admissible w.r.t.\ $(2+\rho^{0,L}+\underline{\rho}^R,\underline{x}^R)$. Let $H_\gamma^R$ be  the right part of $\bbH\backslash\gamma$. Sample an $\SLE_\kappa(\underline{\rho})$ process $\eta_0$ in $H_\gamma^R$ with force points $\underline{x}$. Then there exists some constant $Z(\gamma,\underline{\rho},\underline{x})$ such that the law of the time reversal of $J(\eta_0)$ is equal to $\frac{1}{Z(\gamma,\underline{\rho},\underline{x})}$ times  $\wt{\SLE}_\kappa(\underline{\hat{\rho}};\underline{\hat{\alpha}})$ in $J(H_\gamma^R)$ with the force points $\underline{\hat{x}}$. 
\end{corollary}

\begin{proof}
	{We apply Lemma~\ref{lem:onesided+} within (finitely many) connected components whose boundaries contain the force points $\underline{x}$, and apply Theorem~\hyperref[thm:igII]{A} for rest of the connected components (where there are no constants).  Then the constant $Z(\gamma,\underline{\rho},\underline{x})$ is now a (finite) product of the corresponding constants in each of the connected components.}
\end{proof}

The next lemma states that in some sense, the constant $Z$ in Lemma~\ref{lem:onesided+} does not depend on the choice of $\underline{x}$. {This follows by comparing the two ways of viewing the marginal law of the $\hat{\eta}_1$ above: directly applying Lemma~\ref{lem:onesided+}, and applying Proposition~\ref{prop:commutation} to the pair $(\hat{\eta}_1,\hat{\eta}_2)$.}
\begin{lemma}\label{lem:partition-function}
	In the setting of Corollary~\ref{cor}, the constant $Z(\gamma,\underline{\rho},\underline{x})$ does not depend on $\gamma$.
\end{lemma}

\begin{proof}
	Let $\eta_1$, $\eta_2$, $\hat{\eta}_1$, $\hat{\eta}_2$ be as in the proof of Lemma~\ref{lem:onesided+} and {$\bbP$ be the corresponding background probability measure}.  Let $\hat{\mathcal{P}}$ be the probability measure describing the law of $({\eta}_1^0,\eta_2^0)$ where $\eta_2^0$ is an $\SLE_\kappa(\underline{\hat{\rho}}^L)$ with force points $\underline{\hat{x}}^L$ and $\eta_1^0$ is an $\SLE_\kappa(\rho^{0,L},-\rho^{0,L}-2)$ in the right part of $\bbH\backslash\eta_2^0$ with force points $(0^-;0^+)$. Then {by applying Lemma~\ref{lem:onesided+} (to $\eta_2$) and   Proposition~\ref{prop:commutation}}, the law of $(\hat{\eta}_1,\hat{\eta}_2)$ is absolutely continuous w.r.t.\ $\hat{\mathcal{P}}$ with Radon-Nikodym derivative 
	\begin{equation}\label{eq:partition-RN}
		\frac{1}{Z(\underline{\rho}^R;\underline{x}^R)}\prod_{i=1}^\ell |\hat{x}^{i,L}\cdot\psi_{\hat{\eta}_2}'(\hat{x}^{i,L})|^{\frac{\hat{\rho}^{i,R}(\kappa-4)}{2\kappa}}.
	\end{equation} 
	{On the other hand, since the conditional law of $\eta_2$ given $\eta_1$ is $\SLE_\kappa(\rho^{0,L};\underline{\rho}^R)$, it follows from the definition of the constant $Z(\gamma,\underline{\rho},\underline{x})$ that the conditional law of $\hat\eta_2$ given $\hat\eta_1$ is $\frac{1}{Z({{\eta}}_1,\underline{\rho},\underline{x})}\wt{\SLE}_\kappa(\underline{\hat{\rho}};\underline{\hat{\alpha}})$ in the left part of $\bbH\backslash\hat\eta_1$. {Moreover, we know from Proposition~\ref{prop:commutation} that the marginal law of $\eta_1$ is $\SLE_\kappa(-\rho^{0,L}-2;2+\rho^{0,L}+\underline{\rho}^R)$ with force points $0^-;\underline{x}^R$.} By Lemma~\ref{lem:onesided+}, 
		there exists some constant $Z_1:=Z((-\rho^{0,L}-2;2+\rho^{0,L}+\underline{\rho}^R),\underline{x}^R)$ such that the marginal law of the  curve $\hat{\eta}_1$ is $1/Z_1$ times the $\wt{\SLE}_\kappa(\rho^{0,L}+2+\underline{\hat{\rho}}^L;-\rho^{0,L}-2;\underline{\hat{\alpha}}^L;0)$ with force points $(\underline{\hat{x}}^L;0^+)$. Together with Proposition~\ref{prop:commutation}, we infer that the law of $(\hat{\eta}_1,\hat{\eta}_2)$ is absolutely continuous w.r.t.\ $\hat{\mathcal{P}}$ with Radon-Nikodym derivative}
	\begin{equation}\label{eq:partition-RN-2}
		\frac{1}{Z(\eta_1,\underline{\rho},\underline{x})Z((-\rho^{0,L}-2;2+\rho^{0,L}+\underline{\rho}^R),\underline{x}^R)}\prod_{i=1}^\ell |\hat{x}^{i,L}\cdot\psi_{\hat{\eta}_2}'(\hat{x}^{i,L})|^{\frac{\hat{\rho}^{i,R}(\kappa-4)}{2\kappa}}.
	\end{equation} 
	%Therefore as we fix $\hat{\eta}_1$ and integrate~\eqref{eq:partition-RN} over $\hat{\eta}_2$ w.r.t.\ $\hat{\mathcal{P}}$,} by Proposition~\ref{prop:commutation} and the definition of the constant $Z(\gamma,\underline{\rho},\underline{x})$, {the law of $(\hat{\eta}_1,
	% \hat{\eta}_2)$ can be produced by (i) sampling $\hat{\eta}_1$ from the measure $\wt{\SLE}_\kappa(\rho^{0,L}+2+\underline{\hat{\rho}}^L;-\rho^{0,L}-2;\underline{\hat{\alpha}}^L)$ and weight its law by $\frac{Z(\overline{\hat{\eta}}_1,\underline{\rho},\underline{x})}{Z(\underline{\rho}^R;\underline{x}^R)}$, where $\overline{\hat{\eta}}_1$ is the time reversal of $J\circ\hat{\eta}_1$ and (ii) sample $\hat{\eta}_2$ in the left component of $\bbH\backslash\hat{\eta}_1$ from the probability measure $\frac{1}{Z(\overline{\hat{\eta}}_1,\underline{\rho},\underline{x})}\wt{\SLE}_\kappa(\underline{\hat{\rho}};\underline{\hat{\alpha}})$. This implies that} the marginal law of $\hat{\eta}_1$ {under $\bbP$} is $\frac{Z(\eta_1,\underline{\rho},\underline{x})}{Z(\underline{\rho}^R;\underline{x}^R)}$$\wt{\SLE}_\kappa(\rho^{0,L}+2+\underline{\hat{\rho}}^L;-\rho^{0,L}-2;\underline{\hat{\alpha}}^L)$ {(Recall that $\eta_1$ is the time reversal of $J\circ\hat\eta_1$).} %{On the other hand, by Lemma~\ref{lem:onesided+}, the marginal law of $\hat{\eta}_1$ is $\frac{1}{Z((-\rho^{0,L}-2;2+\rho^{0,L}+\underline{\rho}^R),\underline{x}^R)}\wt{\SLE}_\kappa(\rho^{0,L}+2+\underline{\hat{\rho}}^L;-\rho^{0,L}-2;\underline{\hat{\alpha}}^L)$.} 
It then follows by comparing~\eqref{eq:partition-RN} with~\eqref{eq:partition-RN-2} that  $Z(\eta_1,\underline{\rho},\underline{x}) =\frac{Z(\underline{\rho}^R;\underline{x}^R)}{Z((-\rho^{0,L}-2;2+\rho^{0,L}+\underline{\rho}^R),\underline{x}^R)}$ a.s..  We condition on the positive probability event $E_\e$ for $(\eta_1,\gamma)$ as in Lemma~\ref{lem:sle>0prob}. Then the domain $D_{\eta_1}^{j,R}$ is converging in Caratheodory topology to $D_{\gamma}^{j,R}$ as $\e\to 0$, and the claim follows {from a similar argument as in the end of the proof of Lemma~\ref{lem:onesided+} via $\SLE_\kappa(\underline{\rho})$ continuity over the location of force points}.
\end{proof}

%Now we are ready to prove Theorem~\ref{thm:main} for $\kappa\in (0,4)$.
\begin{proposition}\label{prop:kappa<4}
Theorem~\ref{thm:main} holds for $\kappa\in (0,4)$.
\end{proposition}
\begin{proof}
{The proof is organized as follows. We first construct a pair of reversed curves $(\hat\eta_1,\hat\eta_2)$ by Proposition~\ref{prop:commutation}, and then apply Lemma~\ref{lem:onesided+} to get the conditional laws of $\eta_i := \mathcal{R}(\eta_i)$ given $\eta_j$ for $1\le i\neq j \le 2$.  Finally we apply Lemma~\ref{lem:resampling-a} to identify the law of the forward curves $(\eta_1,\eta_2)$ with the usual $\SLE_\kappa(\underline{\rho})$.  }	

Let {$\hat\eta_1$ be an $\SLE_\kappa(\underline{\hat\rho})$ process in $\bbH$ from 0 to $\infty$ with force points $\underline{\hat x}$}. Pick $0<a<2$ such that the weights $\underline{\tilde{\rho}} = (a+\underline{\hat\rho}^L;-a+\underline{\hat\rho}^R)$ also satisfy the bound~\eqref{eq:continuation-threshold}. Conditioned on $\hat\eta_1$, sample an $\SLE_\kappa(a-2;\underline{\tilde{\rho}}^R)$ process $\hat\eta_2$ in the right component of $\bbH\backslash\hat\eta_1$ with force points $(0^-;\underline{\hat x}^R)$. For $\e>0$, let ${D}_{\e} = \cup_{q\in\{L,R\}}\cup_{i\ge1}B({x}^{i,L},\e)$, $\hat D_{\e} = J({D}_{\e})$ and suppose $\e$ is sufficiently small such that $0,\infty\notin D_{\e}$. Let $\hat F_{\e}$ be the event where $(\hat\eta_1,\hat\eta_2)$ are disjoint from $\overline{\hat D}_\e$. Define the conformal maps $\psi_{\hat\eta_j}^{i,q}$ as in Proposition~\ref{prop:commutation}. Let $\hat{\mathcal{P}}$ be the law of $(\hat\eta_1,\hat\eta_2)$, and define the probability measure $\mathcal{Q}_{\e}$ on pairs of non-crossing simple curves from 0 to $\infty$ by
\begin{equation}\label{eq:twocurveRN}
	\frac{d\mathcal{Q}_{\e}}{d\hat{\mathcal{P}}}(\hat{\eta}_1,\hat{\eta}_2) =\frac{1}{Z_\e} 1_{\hat F_{\e}}\cdot\prod_{i=1}^k|\hat x^{i,L}\cdot(\psi_{\hat{\eta}_1}^{i,L})'(\hat x^{i,L})|^{\frac{\hat\rho^{i,L}(\kappa-4)}{2\kappa}}\cdot\prod_{j=1}^\ell|\hat x^{j,R}\cdot(\psi_{\hat\eta_2}^{j,R})'(\hat x^{j,R})|^{\frac{\hat\rho^{j,R}(\kappa-4)}{2\kappa}}
\end{equation}
{where $Z_\e$ is the
	normalizing constant making $\mathcal{Q}_\e$ a probability measure.}
Observe that on the event $\hat F_{\e}$, by Koebe's 1/4 theorem, there exists some constant $M$ depending only on $\e$ and $\underline{x}$ such that $1/M<|(\psi_{\hat\eta_j}^{i,q})'(\hat x^{i,q})|<M$, which implies that the constant $Z_\e$ is well-defined. Moreover, by Proposition~\ref{prop:commutation}, under the measure $\mathcal{Q}_{\e}$, conditioned on $\hat\eta_1$, $\hat\eta_2$ is an $\wt{\SLE}_\kappa(a-2;\underline{\tilde{\rho}}^R;\underline{\hat\alpha}^R)$ process in the right component of $\bbH\backslash\hat\eta_1$ with force points $(0^-;\underline{x}^R)$ conditioned on not hitting $\overline{\hat D}_{\e}$, while conditioned on $\hat\eta_2$, $\hat\eta_1$ is an $\wt{\SLE}_\kappa(\underline{\hat\rho}^L;a-2;\underline{\hat\alpha}^L)$ process in the left component of $\bbH\backslash\hat\eta_2$ with force points $(\underline{\hat x}^L;0^+)$ conditioned on not hitting $\overline{\hat D}_\e$, where $\hat\alpha^{i,q}=\frac{\hat\rho^{i,q}(\kappa-4)}{2\kappa}$.

For $i=1,2$, let ${\eta}_i=\mathcal{R}(J\circ\hat\eta_i)$. Then by Lemma~\ref{lem:onesided+}, under the measure $\mathcal{Q}_\e$, given ${\eta}_1$, the conditional law of ${\eta}_2$ is the $\SLE_\kappa(-a+\underline{{\rho}}^L;a-2)$ process in the left component of $\bbH\backslash{\eta}_1$ with force points $(\underline{{x}}^L;0^+)$ conditioned on not hitting $\overline{{D}}_{\e}$, and the conditional law of ${\eta}_1$ given ${\eta}_2$ is the  $\SLE_\kappa(a-2;\underline{{\rho}}^R)$ process in the right component of $\bbH\backslash{\eta}_2$ with force points $(0^-;\underline{{x}}^R)$ conditioned on not hitting $\overline{{D}}_\e$. Let $\tilde\eta_1$ be an independent $\SLE_\kappa(\underline{{\rho}})$ process in $\bbH$ from 0 to $\infty$ with force points $\underline{{x}}$, and sample an  $\SLE_\kappa(-a+\underline{{\rho}}^L;a-2)$ process $\tilde{\eta_2}$ in the left component of $\bbH\backslash\tilde{\eta}_1$ with force points $(\underline{{x}}^L;0^+)$. Therefore $(\eta_1,\eta_2)$ and $(\tilde\eta_1,\tilde\eta_2)$ satisfy the same resampling properties as in Lemma~\ref{lem:resampling-a}. Then by Lemma~\ref{lem:resampling-a}, the joint law of $({\eta}_1,{\eta}_2)$ agrees with that of $(\tilde{\eta}_1,\tilde{\eta}_2)$ conditioned on $\{\tilde{\eta}_1\cap \overline{{D}}_\e = \tilde \eta_2\cap \overline{{D}}_{\e}=\emptyset \}$. In particular, the marginal law of ${\eta}_1$ under $\mathcal{Q}_\e$ is the $\SLE_\kappa(\underline{{\rho}})$ process conditioned on not hitting $\overline{{D}}_\e$ and weighted by the probability where the $\SLE_\kappa(-a+\underline{{\rho}}^L;a-2)$ process in the left component of $\bbH\backslash{\eta}_1$ is disjoint from $\overline{{D}}_\e$.

{On the other hand, by Proposition~\ref{prop:commutation}, a sample $(\hat\eta_1,\hat\eta_2)$ from $\mathcal{Q}_\e$ can be produced by (i) sampling $\hat\eta_1$ from $\wt{\SLE}_\kappa(\underline{\hat\rho};\underline{\hat\alpha})$ process conditioned on not hitting $\overline{\hat{D}}_\e$ (ii) weighting the law of $\hat\eta_1$ by $Z_0(\hat\eta_1)$, where $Z_0(\hat\eta_1)$ is the measure of an $\wt{\SLE}_\kappa(a-2;\underline{\tilde{\rho}}^R;0;\underline{\hat\alpha}^R)$ process in the right component of $\bbH\backslash\hat\eta_1$ with force points $(0^-;\underline{x}^R)$ being disjoint from $\overline{\hat{D}}_\e$ and (iii) sampling an $\wt{\SLE}_\kappa(a-2;\underline{\tilde{\rho}}^R;0;\underline{\hat\alpha}^R)$ process $\hat\eta_2$ in the right component of $\bbH\backslash\hat\eta_1$ with force points $(0^-;\underline{x}^R)$ conditioned on not hitting $\overline{\hat{D}}_\e$. Meanwhile, by Lemma~\ref{lem:onesided+}, $Z_0(\hat\eta_1)$ is equal to $Z(\eta_1,(-a+\underline{{\rho}}^L;a-2),\underline{x}^R)$ times the probability of an $\SLE_\kappa(-a+\underline{{\rho}}^L;a-2)$ process in the left component of $\bbH\backslash{\eta}_1$ being disjoint from $\overline{{D}}_\e$.} %from the Radon-Nikodym derivative~\eqref{eq:twocurveRN}, by Proposition~\ref{prop:commutation}, Lemma~\ref{lem:onesided+} and Lemma~\ref{lem:partition-function} (since ${\eta}_1$ is admissible w.r.t. $(\underline{{\rho}}^L;\underline{{x}}^L)$), the marginal law of $\hat\eta_1$ under the measure $\mathcal{Q}_{\e}$ is the $\wt{\SLE}_\kappa(\underline{\hat\rho};\underline{\hat\alpha})$ process conditioned on not hitting $\overline{\hat{D}}_\e$ and weighted by the probability where the $\SLE_\kappa(-a+\underline{{\rho}}^L;a-2)$ process in the left component of $\bbH\backslash{\eta}_1$ is disjoint from $\overline{{D}}_\e$. Since 
By Lemma~\ref{lem:sle>0prob}, the latter probability is positive for any fixed ${\eta}_1$, while by Lemma~\ref{lem:partition-function}, the constant $Z(\eta_1,(-a+\underline{{\rho}}^L;a-2),\underline{x}^R)$ is independent of $\eta_1$. Therefore by comparing the marginal laws of $\eta_1$ and $\hat\eta_1$,  we conclude that under $J(z)=-1/z$, the time reversal of an $\SLE_\kappa(\underline{{\rho}})$ process with force point $\underline{{x}}$ conditioned on not hitting $\overline{{D}}_\e$ agree with the $\wt{\SLE}_\kappa(\underline{\hat\rho};\underline{\hat\alpha})$ process with force point $\underline{\hat x}$ conditioned on not hitting $\overline{\hat{D}}_\e$. Since $\e>0$ can be arbitrarily small, the claim therefore follows. % by swapping $(\underline{\hat{\rho}};\underline{\hat{x}})\leftrightarrow(\underline{\rho};\underline{x})$.
\end{proof}

For $\kappa\in (4,8)$, the argument is based on the following \emph{SLE duality} argument, which follows from~\cite[Theorem 5.1]{zhan2008duality} and~\cite[Theorem 1.4, Proposition 7.30]{MS16a}.  
\begin{proposition}\label{prop:duality}
Let $\kappa\in (4,8]$, $\tilde{\kappa}=\frac{16}{\kappa}$ and $\underline{\rho}$ satisfying~\eqref{eq:continuation-threshold} and $x^{k,L}<...<x^{0,L}=0^-<x^{0,R}=0^+<...<x^{\ell,R}$. Let $\rho^{k+1,L}=-\sum_{i=0}^k\rho^{i,L}$ and $\rho^{\ell+1,R}=-\sum_{i=0}^\ell\rho^{i,R}$. Let $\hat{\rho}^{i,L} = -\rho^{\ell+1-i,R}$ for $0\le i\le \ell$ and $\hat{\rho}^{j,R} = -\rho^{k+1-j,L}$ for $0\le j\le k$. Let $\eta'$ be an $\SLE_{\kappa}(\underline{\rho})$ process in $\bbH$ from 0 to $\infty$.  Then the left boundary $\eta_L$ of $\eta'$ is an $\SLE_{\tilde{\kappa}}(\frac{\tilde{\kappa}}{2}-2+\frac{\tilde{\kappa}}{4}\underline{\hat{\rho}}^{L};\tilde{\kappa}-4+\frac{\tilde{\kappa}}{4}\underline{\hat{\rho}}^{R})$ process from $\infty$ to 0 with force points $\hat{\underline{x}}:=(+\infty, x^{\ell,R},...,x^{1,R};-\infty, x^{k,L},...,x^{1,L})$, and the right boundary $\eta_R$ of $\eta'$ is an $\SLE_{\tilde{\kappa}}({\tilde{\kappa}}-4+\frac{\tilde{\kappa}}{4}\underline{\hat{\rho}}^{L};\frac{\tilde{\kappa}}{2}-2+\frac{\tilde{\kappa}}{4}\underline{\hat{\rho}}^{R})$ process from $\infty$ to 0 with force points $\hat{\underline{x}}$. Moreover, conditioned on $\eta_L$ and $\eta_R$, $\eta'$ is an $\SLE_\kappa(\frac{\kappa}{2}-4;\frac{\kappa}{2}-4)$ process independently in each connected component of $\bbH\backslash(\eta_L\cup\eta_R)$ between $\eta_L$ and $\eta_R$, and conditioned on $\eta_L$, $\eta_R$ is an $ \SLE_{\tilde{\kappa}}({\tilde{\kappa}}-4+\frac{\tilde{\kappa}}{4}\underline{\hat{\rho}}^{L};-\frac{\tilde{\kappa}}{2})$ process in $\bbH\backslash\eta_L$ to the right of $\eta_L$ from $\infty$ to 0 with force points $(+\infty, x^{\ell,R},...,x^{1,R};-\infty)$.
\end{proposition}

\begin{proposition}\label{prop:kappa>4}
Theorem~\ref{thm:main} holds for $\kappa\in (4,8]$.
\end{proposition}
\begin{proof}
Let $\eta'$ be an $\SLE_{\kappa}(\underline{\rho})$ process in $\bbH$ from 0 to $\infty$ with force point $\underline{x}$, $\eta_L,\eta_R$ be its left and right boundary, and $\mathcal{R}(\eta')$ be the time-reversal of $\eta'$. Let $\tilde{\kappa}=\frac{16}{\kappa}$, $\tilde{x}^{i,L}=x^{\ell+1-i,R}$ for $0\le i\le \ell$ and $\tilde{x}^{j,R}=x^{k+1-j,L}$ for $0\le j\le k$, where $x^{k+1,L}=-\infty$ and $x^{\ell+1,R}=+\infty$. Let $\hat{\rho}^{i,q}$ be as in the statement of Proposition~\ref{prop:duality}. {Note that $\underline{\hat{x}} = J(\underline{\tilde{x}})$.}
Then by Proposition~\ref{prop:kappa<4} and Proposition~\ref{prop:duality}, the law of the left boundary $\mathcal{R}(\eta_L)$ of $\mathcal{R}(\eta')$ is proportional to the $\wt{\SLE}_{\tilde{\kappa}} (\tilde{\kappa}-4+\frac{\tilde{\kappa}}{4}\underline{\rho}^L;\frac{\tilde{\kappa}}{2}-2+\frac{\tilde{\kappa}}{4}\underline{\rho}^R;\underline{\alpha}^L;\underline{\alpha}^R)$ from 0 to $\infty$ with force points $\underline{x}$ and $\alpha^{i,q} = \frac{\tilde{\kappa}}{4}\rho^{i,q}\cdot\frac{\tilde{\kappa}-4}{2\tilde{\kappa}} = -\rho^{i,q}\frac{4-\kappa}{2\kappa}$. Likewise, conditioned on $\mathcal{R}(\eta_L)$, the law of the right boundary $\mathcal{R}(\eta_R)$ of $\mathcal{R}(\eta')$ is $\frac{1}{Z({\eta}_L)}\wt{\SLE}_{\tilde{\kappa}}(-\frac{\tilde{\kappa}}{2};{\tilde{\kappa}}-4+\frac{\tilde{\kappa}}{4}\underline{{\rho}}^{R};0;\underline{\alpha}^R)$ process from 0 to $\infty$ to the right of $\bbH\backslash\mathcal{R}(\eta_L)$ with force points $(0^-;\underline{x}^R)$. Moreover, by Lemma~\ref{lem:partition-function}, since $-\frac{\tilde{\kappa}}{2}<0$, the constant $Z({\eta}_L)$ does not depend on $\eta_L$.

On the other hand, let $\tilde{\eta}'$ be an $\wt{\SLE}_\kappa(\underline{\hat{\rho}};\underline{\hat{\alpha}})$ process in $\bbH$ from $\infty$ to 0 with force points $\underline{\tilde{x}}$, and $\hat{\alpha}^{i,q}=\frac{\hat{\rho}^{i,q}(\kappa-4)}{2\kappa}$. Note that by definition, for each point $x^{i,q}$, its assigned power parameter is precisely $\alpha^{i,q}$. %Let $\tilde{D}_{i,q}$ be the connected component of $\bbH\backslash\eta'$ with $\hat{x}^{i,q}$ on the boundary, and $\hat{\psi}_{i,q}:\tilde{D}_{i,q}\to\bbH$ sending $\hat{x}^{i,q}$ to 1 and the first (resp.\ last) point on $\partial D_{i,q}$ traced by $\tilde{\eta}'$ to $\infty$ (resp.\ 0). 
Then by Proposition~\ref{prop:duality}, the law of the left and right boundaries $(\tilde{\eta}_L,\tilde{\eta}_R)$ of $\tilde{\eta}'$ can be produced by the following procedure:

(i) Sample an ${\SLE}_{\tilde{\kappa}} (\tilde{\kappa}-4+\frac{\tilde{\kappa}}{4}\underline{\rho}^L;\frac{\tilde{\kappa}}{2}-2+\frac{\tilde{\kappa}}{4}\underline{\rho}^R)$ process $\eta_1$ from 0 to $\infty$ with force points $\underline{x}$, and given $\eta_1$, sample an ${\SLE}_{\tilde{\kappa}}(-\frac{\tilde{\kappa}}{2};{\tilde{\kappa}}-4+\frac{\tilde{\kappa}}{4}\underline{{\rho}}^{R})$ process $\eta_2$ to the right of $\bbH\backslash\eta_1$ with force points $(0^-,\underline{x}^R)$;

(ii) For $i\ge 1$ and $j=1,2$, let $D_{\eta_j}^{i,q}$ be the connected component of $\bbH\backslash\eta_j$ with $x^{i,q}$ on the boundary, $\sigma_{\eta_j}^{i,q}$ (resp.\ $\xi_{\eta_j}^{i,q}$) be the first (resp.\ last) point on $\partial D_{\eta_j}^{i,q}$ traced by $\eta_j$, and $\psi_{\eta_j}^{i,q}:D_{\eta_j}^{i,q}\to\bbH$ be the conformal map sending $(\sigma_{\eta_j}^{i,q},x^{i,q},\xi_{\eta_j}^{i,q})$ to $(0,\pm 1,\infty)$ where we take the + sign when $q=R$;

(iii) Weight the law of $(\eta_1,\eta_2)$ by 
$$\prod_{i=1}^k|x^{i,L}\cdot(\psi_{\eta_1}^{i,L})'(x^{i,L})|^{{\alpha}^{i,L}}\cdot \prod_{i=1}^\ell |x^{i,R}\cdot(\psi_{\eta_2}^{i,R})'(x^{i,R})|^{{\alpha}^{i,R}}.$$

We conclude by Proposition~\ref{prop:commutation}  that up to a finite multiplicative constant, the law of $(\mathcal{R}(\eta_L),\mathcal{R}(\eta_R))$ agrees with that of $(\tilde{\eta}_L,\tilde{\eta}_R)$. Moreover, by Proposition~\ref{prop:duality} and Theorem~\hyperref[thm:igII]{A}, conditioned on $(\mathcal{R}(\eta_L),\mathcal{R}(\eta_R))$, $\mathcal{R}(\eta')$ is an  $\SLE_\kappa(\frac{\kappa}{2}-4;\frac{\kappa}{2}-4)$ process independently in each connected component of $\bbH\backslash(\mathcal{R}(\eta_L)\cup\mathcal{R}(\eta_R))$ between $\mathcal{R}(\eta_L)$ and $\mathcal{R}(\eta_R)$, which agrees with the conditional law of $\tilde{\eta}'$ given $(\tilde{\eta}_L,\tilde{\eta}_R)$. Therefore the law of $\mathcal{R}(\eta')$ agrees up to a multiplicative constant with that of $\tilde{\eta}'$, which finishes the proof of Theorem~\ref{thm:main} for $\kappa\in (4,8])$.
\end{proof}

\begin{proof}[Proof of Theorem~\ref{thm:main}]
{The theorem follows directly by applying Proposition~\ref{prop:kappa<4} for $\kappa\in(0,4)$, Proposition~\ref{prop:kappa>4} for $\kappa\in(4,8]$ along with~\cite[Theorem 1.1.6]{wang2017level} for $\kappa=4$}.
\end{proof}

\appendix
\section{Proof of Lemma~\ref{lem:resampling}}\label{sec:appendix}
In this section, we give an alternative proof of Lemma~\ref{lem:resampling} based on Lemma~\ref{lem:sle>0prob} and the irreducibility of Markov chain argument from~\cite{Meyn-Tweedie}. Let $(X,\mathcal{F})$ be a state space where $\mathcal{F}$ is a $\sigma$-algebra. For a Markov chain $\{X_n\}_{n\ge0}$ and a measure $\varphi$ on $(X,\mathcal{F})$, if for any $x\in X$ and $A\in\mathcal{F}$, $\bbP(X_n\in A \ \text{for some }n|X_0=x)>0$ whenever $\varphi(A)>0$, then  $\{X_n\}_{n\ge0}$ is said to be \emph{$\varphi$-irreducible}. By~\cite[Theorem 4.0.1]{Meyn-Tweedie}, there exists a unique \emph{maximal irreducibility} measure $\psi$ on $(X,\mathcal{F})$ such that $\{X_n\}_{n\ge0}$ is \emph{$\psi$-irreducible}. Then~\cite[Proposition 10.1.1, Theorem 10.0.1]{Meyn-Tweedie} tells us that a $\psi$-irreducible Markov chain with an invariant probability measure is recurrent, and thus admit a unique invariant measure.

\begin{figure}[htb]
\centering
\begin{tabular}{ccc} 
	\includegraphics[scale = 0.6]{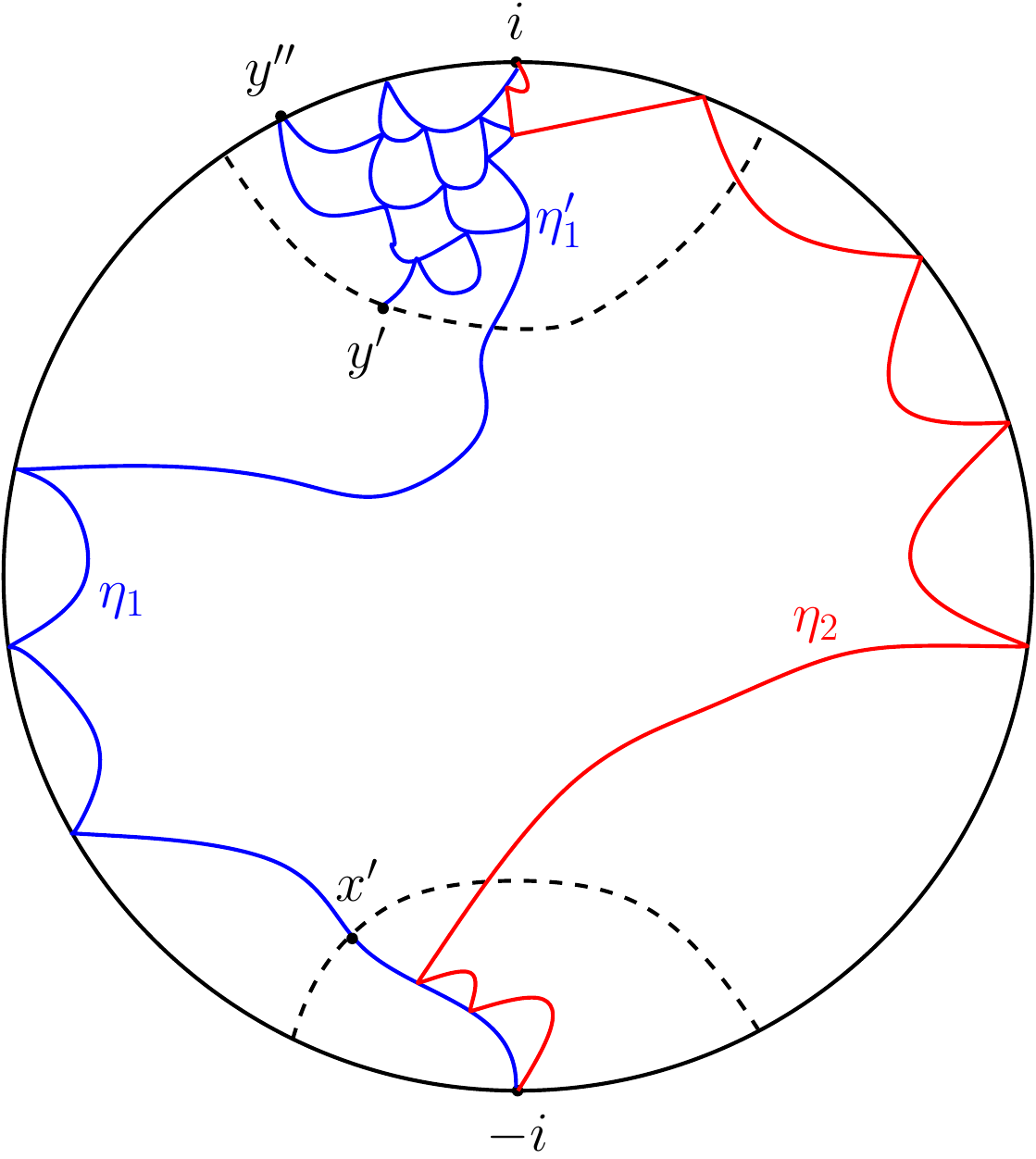}
	&  
	\includegraphics[scale = 0.6]{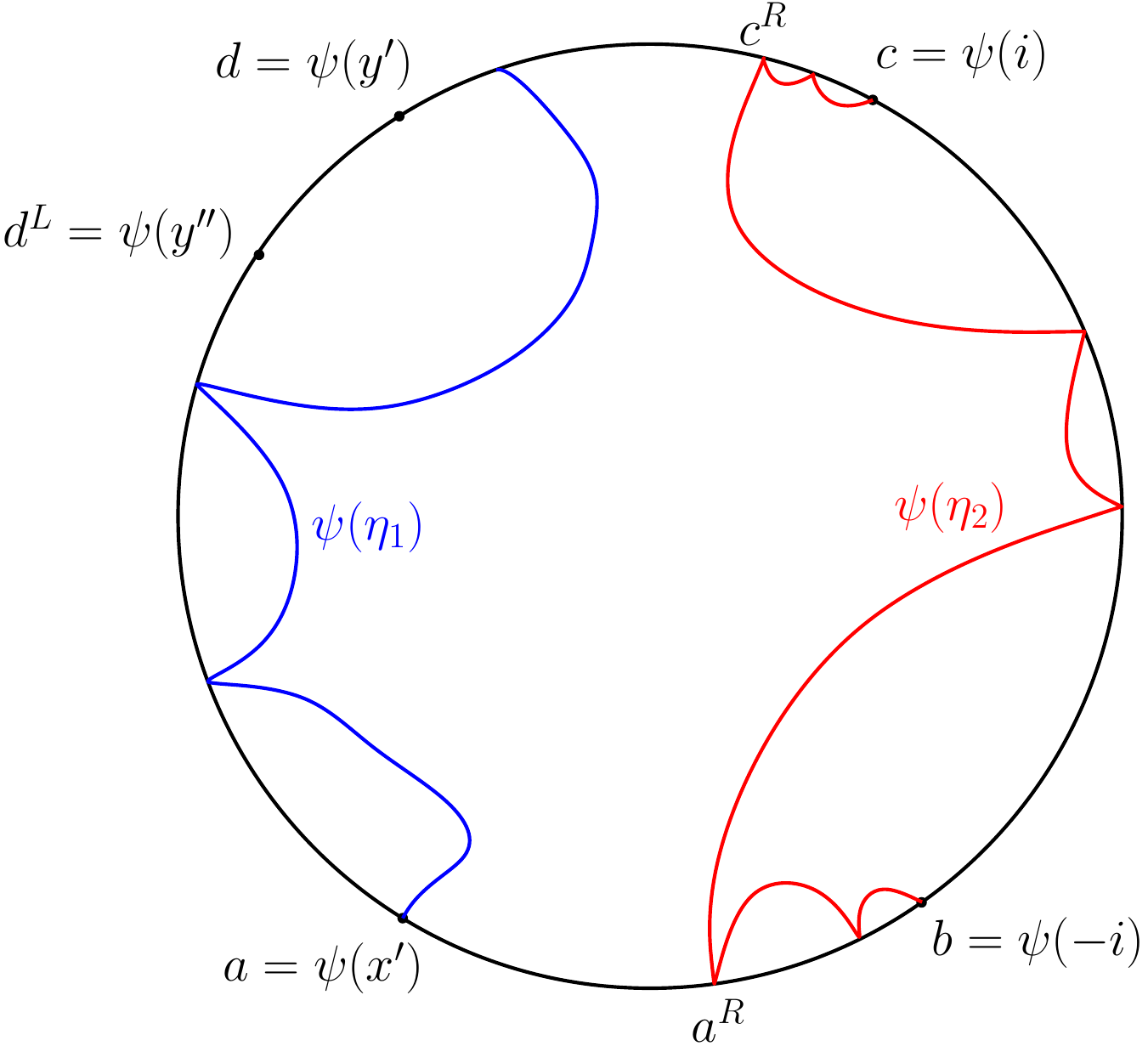}
\end{tabular}
\caption{First step of the proof of Lemma~\ref{lem:resampling}, where we first run $\eta_1$ and its associated counterflowline $\eta_1'$ until they hit $\partial B(-i,\e)$ and $\partial B(i,\e)$ as in~\cite[Proof of Theorem 4.1]{MS16b} and map back to $\bbD$ by a conformal map $\psi$. Given $\psi(\eta_1)$, $\psi(\eta_2)$ is an $\SLE_\kappa(\rho-2;\underline{\rho}^R)$ process in the right component of $\bbD\backslash\psi(\eta_1)$, while given $\psi(\eta_2)$, $\psi(\eta_1)$ is an $\SLE_\kappa(\underline{\rho}^L,\frac{\kappa}{2}-2-\overline{\rho}^L;\rho-2,\kappa-2-\rho)$ in the left component of  $\bbD\backslash\psi(\eta_2)$ with the force points on the right located at $a^R,c^R$. Note that since $\eta_1$ a.s.\ merges into $\eta_1'$, $\psi(\eta_1)$ a.s.\ terminates between $c$ and $d$.}
\label{fig:resamplingpf1}
\end{figure}

Without loss of generality, we take a conformal map $\bbH\to\bbD$ and assume that we are in the setting where $\eta_1,\eta_2$ are continuous curves in $\bbD$ from $-i$ to $i$ such that for $\{j,k\}=\{1,2\}$, given one curve $\eta_k$, the other curve is the $\SLE_\kappa(\underline{\rho}^j)$ process in $\bbD\backslash\eta_k$ (recall the definition of $\SLE_\kappa(\underline{\rho})$ processes in non-simply connected domains in Section~\ref{subsec:pre-ig}) as in the statement of Lemma~\ref{lem:resampling}. By an identical argument of the first step of the proof of~\cite[Theorem 4.1]{MS16b} (i.e., draw counterflowlines $\eta_1'$ by SLE duality, run $\eta_1,\eta_1'$ for a small amount of time and look at the remaining parts of $\eta_1,\eta_2$), we may work on the case where the starting and ending points of $\eta_1,\eta_2$ are distinct. To be more precise, let $a,b,c,d\in\partial\bbD$ be 4 points in counterclockwise order,  ${x}^{0,L},...,x^{k,L},d^L$ be some marked points on the $\overset{{\frown}}{da}$ arc of $\partial \bbD$, and  ${x}^{0,R},...,x^{\ell,R}$ be some marked points on the $\overset{{\frown}}{bc}$ arc of $\partial \bbD$ with $x^{0,R}=a^+$. Let $X$ be the space of non-crossing continuous curves $(\gamma_1,\gamma_2)$ connecting $(a,b)$ with $(d,c)$ in $\overline{\bbD}$ such that $\gamma_1$ (resp.\ $\gamma_2$) is disjoint from $\underline{x}^R\cup \{c\}$ (resp.\ $\underline{x}^L\cup\{d\}$) and does not trace any segment of the arc $\overset{{\frown}}{bc}$ (resp.\ $\overset{{\frown}}{da}$), and $\mathcal{F}$ be the Borel $\sigma$-algebra on $X$ generated by Hausdorff topology.  We are going to show that there exists at most one probability measure $\mu$ on $(X,\mathcal{F})$ such that, for a sample $(\eta_1,\eta_2)$ from $\mu$, conditioned on $\eta_1$, $\eta_2$ is an $\SLE_\kappa(\rho-2;\underline{\rho}^R)$ curve  in the right component of $\bbD\backslash\eta_1$ with force points $b^-;\underline{x}^R$, and conditioned on $\eta_2$, $\eta_1$ is an $\SLE_\kappa(\underline{\rho}^L,\frac{\kappa}{2}-2-\overline{\rho}^L;\rho-2,\kappa-2-\rho)$ in the left component of $\bbD\backslash\eta_2$ with force points $\underline{x}^L,d^L;a^R,c^R$, where $\overline{\rho}^L = \sum_{i=0}^k\rho^{i,L}$, and $a^R$ (resp.\ $c^R$) is the left most point of $\eta_2\cap \overset{{\frown}}{ab}$ (resp.\  $\eta_2\cap \overset{{\frown}}{cd}$). {See Figure~\ref{fig:resamplingpf1} for an illustration}.

We construct a Markov chain on $(X,\mathcal{F})$ as follows. Let $X_0 = (\eta_1^0,\eta_2^0)\in X$, and for given $n\ge0$ and $X_n = (\eta_1^n,\eta_2^n)$, we first uniformly pick $i\in\{1,2\}$ and sample $\eta_i^{n+1}$ in $\bbD\backslash\eta_{3-i}^n$ from the conditional law induced by $\mu$ as described in the previous paragraph. Let $\eta_{3-i}^{n+1}=\eta_{3-i}^n$ and set $X_{n+1} = (\eta_1^{n+1},\eta_2^{n+1})$. Pick $a',b'$ on the arc $\overset{{\frown}}{ab}$, and $c',d'$ on the arc $\overset{{\frown}}{cd}$ and draw two disjoint simple curves $(\gamma^L,\gamma^R)$ in $\bbD$ connecting  $(a',b')$ with $(d',c')$. Let $\Omega^L$  be left  component of $\bbD\backslash\gamma^L$, and $\Omega^R$  be right component of $\bbD\backslash\gamma^R$. Let $X_\Omega:=\{(\gamma_1,\gamma_2)\in X:\gamma_1\subset \overline{\Omega^L},\gamma_2\subset \overline{\Omega^R}\}$. We are going to show that $\{X_n\}_{n\ge0}$ is $\varphi$-irreducible for $\varphi = \mu|_{X_\Omega}$ and thus admits a unique invariant probability measure, which concludes the proof by~\cite{Meyn-Tweedie}. 

Given $(\eta_1^0,\eta_2^0)\in X$, let $D_{0,R}$, ..., $D_{m,R}$ be the connected components of $\bbD\backslash\eta_1^0$ whose boundary has nonempty intersection with both $\overset{{\frown}}{da}$ and $\overset{{\frown}}{bc}$. Note that the number of such components is finite by the continuity of $\eta_1^0$. Then by applying Lemma~\ref{lem:sle>0prob} in each of $D_{0,R}$, ..., $D_{m,R}$, when we sample $\eta_2^1$ in the right component of $\bbD\backslash\eta_1^0$ from the conditional law induced by $\mu$, there is a positive probability such that $\eta_2^1$ is disjoint from the arc $\overset{{\frown}}{da}$. Under this event, by Lemma~\ref{lem:sle>0prob}, when we sample $\eta_1^2$ from the corresponding conditional law in the left component of $\bbD\backslash\eta_2^1$, there is positive chance that $\eta_1^2$ is disjoint from the arc $\overset{{\frown}}{bc}$ and stays in the domain $\overline{\Omega^L}$. (Note that although $\eta_1^2$ merges with the arc $\overset{{\frown}}{cd}$ before reaching the target $c$, Lemma~\ref{lem:sle>0prob} extends to this setting and is still applicable.) Applying Lemma~\ref{lem:sle>0prob} once more, under this event, when we sample $\eta_2^3$ from the corresponding conditional law in the right component of $\bbD\backslash\eta_1^2$, there is a positive probability that $\eta_2^3$ is contained in $\overline{\Omega^R}$. Therefore we conclude that for any $(\eta_1^0,\eta_2^0)\in X$, $\bbP(X_3\in X_{\Omega}|X_0 = (\eta_1^0,\eta_2^0))>0$. Note that this also implies that $\mu(X_\Omega)>0$. See also Figure~\ref{fig:resamplingpf2}.

Finally, from the GFF flow line local absolute continuity~\cite[Proposition 3.4]{MS16a} and~\cite[Theorem 1.2]{MS16a}, given any curves $\gamma_2$, $\tilde \gamma_2$ in $\overline{\Omega^R}$, when we sample $\eta_1$,  $\tilde\eta_1$ in the left component  of $\bbD\backslash\gamma_2$ and $\bbD\backslash\tilde\gamma_2$ according to the conditional law described by $\mu$, when restricted to the event $\eta_1$,  $\tilde\eta_1$ are contained in $\overline{\Omega^L}$, the laws of  $\eta_1$ and  $\tilde\eta_1$ are mutually absolutely continuous w.r.t.\ each other. In particular, this implies that for any $A\in\mathcal{F}$ with $\mu|_{X_{\Omega}}(A)>0$, $\bbP(X_5\in A|X_0 = (\eta_1^0,\eta_2^0))>0$. This justifies the irreducibility of $\{X_n\}_{n\ge0}$ and thus concludes the proof.
\qed
\begin{figure}[htb]
\centering
\begin{tabular}{ccc} 
	\includegraphics[scale=0.6]{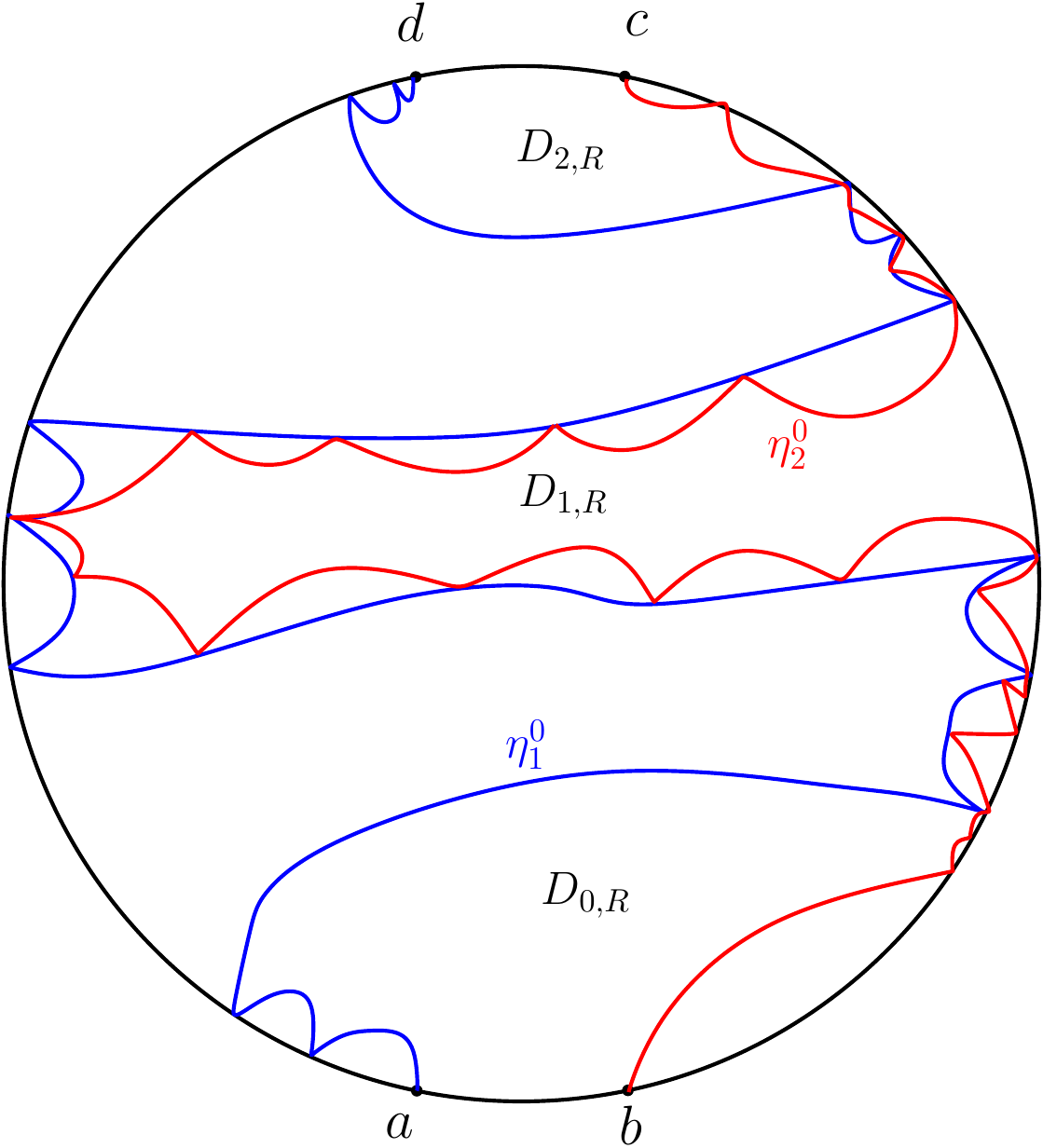}
	& \quad &
	\includegraphics[scale=0.6]{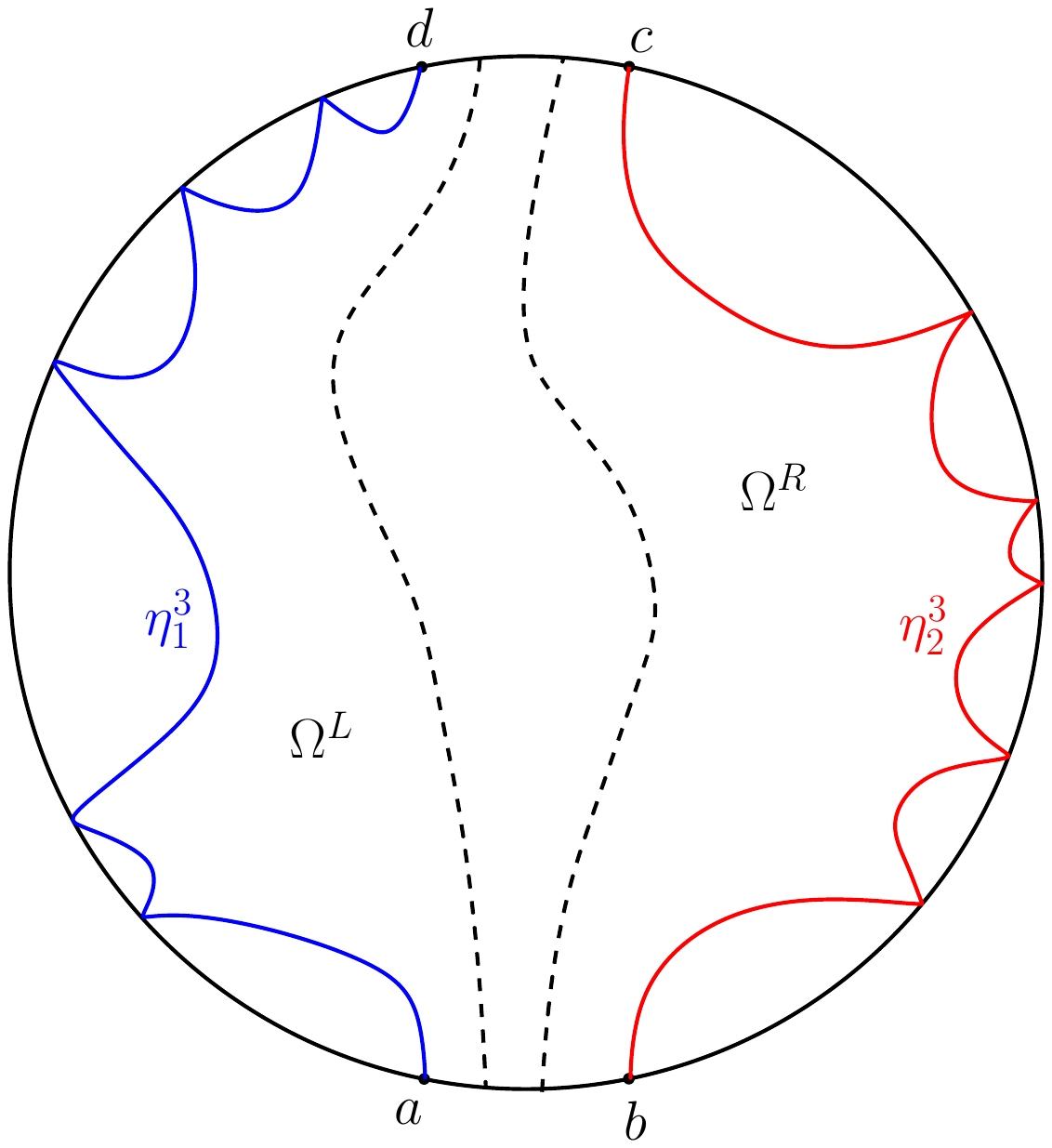}
\end{tabular}
\caption{The Markov chain resampling in Proof of Lemma~\ref{lem:resampling}. \textbf{Left:} Initial phase for the two given curves $(\eta_1^0,\eta_2^0)$, with $D_{0,R},D_{1,R},D_{2,R}$ being the connected components of $\bbD\backslash\eta_1^0$ whose boundary intersects both the $da$ arc and $bc$ arc. \textbf{Right:} By applying Lemma~\ref{lem:sle>0prob} in each of $D_{0,R},D_{1,R},D_{2,R}$, when we sample $\eta_2^1$ in the right component of $\bbD\backslash\eta_1^0$, with positive probability  $\eta_2^1$ is disjoint from the $ad$ arc. It then follows by applying Lemma~\ref{lem:sle>0prob} twice more $\bbP(X_3\in X_{\Omega}|X_0 = (\eta_1^0,\eta_2^0))>0$.}
\label{fig:resamplingpf2}
\end{figure}

%%%%%%%%%%%%%%%%%%%%%%%%%%%%%%%%%%%%%%%%%%%%%%%%%%%%%%%%%%%%%%%%%%%
%%                                                               %%
%% You may add acknowledgments (optional).                       %%
%%                                                               %%
%%%%%%%%%%%%%%%%%%%%%%%%%%%%%%%%%%%%%%%%%%%%%%%%%%%%%%%%%%%%%%%%%%%

\begin{thebibliography}{MSW19}

\bibitem[AHS20]{AHS20}
M.~Ang, N.~Holden, and X.~Sun.
\newblock Conformal welding of quantum disks.
\newblock {\em arXiv preprint arXiv:2009.08389}, 2020.

\bibitem[ASY22]{ASY22}
Morris Ang, Xin Sun, and Pu~Yu.
\newblock Quantum triangles and imaginary geometry flow lines.
\newblock {\em arXiv preprint arXiv:}, 2022.

\bibitem[DMS21]{DMS14}
Bertrand Duplantier, Jason Miller, and Scott Sheffield.
\newblock Liouville quantum gravity as a mating of trees.
\newblock {\em Ast\'{e}risque}, 427, 2021.

\bibitem[Dub09]{Dub09}
J.~Dub\'{e}dat.
\newblock Duality of {S}chramm-{L}oewner {E}volutions.
\newblock {\em Ann. Sci. \'{E}c. Norm. Sup\'{e}r}, 42(5), 2009.

\bibitem[Law08]{Law08}
Gregory~F Lawler.
\newblock {\em Conformally invariant processes in the plane}.
\newblock Number 114. American Mathematical Soc., 2008.

\bibitem[LSW03]{LSW03conformalrestriction}
Gregory Lawler, Oded Schramm, and Wendelin Werner.
\newblock Conformal restriction: the chordal case.
\newblock {\em J. Amer. Math. Soc.}, 16(4):917--955, 2003.

\bibitem[LSW11]{lawler2011conformal}
Gregory~F Lawler, Oded Schramm, and Wendelin Werner.
\newblock Conformal invariance of planar loop-erased random walks and uniform
  spanning trees.
\newblock In {\em Selected Works of Oded Schramm}, pages 931--987. Springer,
  2011.

\bibitem[MS16a]{MS16a}
J.~Miller and S.~Sheffield.
\newblock Imaginary {G}eometry {I}: Interacting {SLE}s.
\newblock {\em Probability Theory and Related Fields}, 164(3-4):553--705, 2016.

\bibitem[MS16b]{MS16b}
Jason Miller and Scott Sheffield.
\newblock {Imaginary geometry II: Reversibility of
  {SLE}$_\kappa(\rho_1;\rho_2)$ for $\kappa\in (0, 4) $}.
\newblock {\em The Annals of Probability}, 44(3):1647--1722, 2016.

\bibitem[MS16c]{IGIII}
Jason Miller and Scott Sheffield.
\newblock Imaginary geometry {III}: reversibility of {SLE}$_\kappa$ for
  $\kappa\in(4,8)$.
\newblock {\em Annals of Mathematics}, pages 455--486, 2016.

\bibitem[MS17]{MS17}
Jason Miller and Scott Sheffield.
\newblock {Imaginary geometry IV: interior rays, whole-plane reversibility, and
  space-filling trees}.
\newblock {\em Probability Theory and Related Fields}, 169(3):729--869, 2017.

\bibitem[MSW19]{MSW19}
Jason Miller, Scott Sheffield, and Wendelin Werner.
\newblock Non-simple {SLE} curves are not determined by their range.
\newblock {\em Journal of the European Mathematical Society}, 22(3):669--716,
  2019.

\bibitem[MT09]{Meyn-Tweedie}
Sean Meyn and Richard~L. Tweedie.
\newblock {\em Markov chains and stochastic stability}.
\newblock Cambridge University Press, Cambridge, second edition, 2009.
\newblock With a prologue by Peter W. Glynn.

\bibitem[MW17]{MW17}
Jason Miller and Hao Wu.
\newblock Intersections of sle paths: the double and cut point dimension of
  sle.
\newblock {\em Probability Theory and Related Fields}, 167(1-2):45--105, 2017.

\bibitem[RS05]{RS05}
S.~Rohde and O.~Schramm.
\newblock Basic properties of {SLE}.
\newblock {\em Ann. of Math.}, 161(2), 2005.

\bibitem[Sch00]{Sch00}
Oded Schramm.
\newblock Scaling limits of loop-erased random walks and uniform spanning
  trees.
\newblock {\em Israel Journal of Mathematics}, 118(1):221--288, 2000.

\bibitem[Sch11]{Sch06ICM}
Oded Schramm.
\newblock Conformally invariant scaling limits: an overview and a collection of
  problems [mr2334202].
\newblock In {\em Selected works of {O}ded {S}chramm. {V}olume 1, 2}, Sel.
  Works Probab. Stat., pages 1161--1191. Springer, New York, 2011.

\bibitem[Smi06]{SmirnovICM}
Stanislav Smirnov.
\newblock Towards conformal invariance of 2{D} lattice models.
\newblock In {\em International {C}ongress of {M}athematicians. {V}ol. {II}},
  pages 1421--1451. Eur. Math. Soc., Z\"{u}rich, 2006.

\bibitem[WW17]{wang2017level}
Menglu Wang and Hao Wu.
\newblock Level lines of {G}aussian {F}ree {F}ield {I}: zero-boundary {GFF}.
\newblock {\em Stochastic Processes and their Applications}, 127(4):1045--1124,
  2017.

\bibitem[Zha08a]{zhan2008duality}
Dapeng Zhan.
\newblock Duality of chordal {SLE}.
\newblock {\em Inventiones mathematicae}, 174(2):309--353, 2008.

\bibitem[Zha08b]{Zhan08reverse}
Dapeng Zhan.
\newblock Reversibility of chordal {SLE}.
\newblock {\em Ann. Probab.}, 36(4):1472--1494, 2008.

\bibitem[Zha22]{Zhan19}
Dapeng Zhan.
\newblock Time-reversal of multiple-force-point {${\rm
  SLE}_\kappa(\underline\rho)$} with all force points lying on the same side.
\newblock {\em Ann. Inst. Henri Poincar\'{e} Probab. Stat.}, 58(1):489--523,
  2022.

\end{thebibliography}


\newpage
\end{document}